\documentclass[11pt]{article}

\usepackage[T1]{fontenc}
\usepackage[utf8]{inputenc}
\usepackage[margin=2.54cm]{geometry}
\usepackage{microtype}
\usepackage{amsmath,amssymb,amsthm}
\numberwithin{equation}{section}
\allowdisplaybreaks
\usepackage{enumitem}
\usepackage{xcolor}
\usepackage{aliascnt}
\usepackage{sectsty}
\sectionfont{\centering}
\usepackage[backref=section]{hyperref}
\hypersetup{
  colorlinks=true,
  linkcolor={blue!85!black},
  citecolor={red!85!black},
  urlcolor={blue!55!black}
}
\usepackage[noabbrev]{cleveref}
\usepackage{tikz}
\usetikzlibrary{arrows.meta,positioning,calc,shapes.geometric}
\theoremstyle{plain}
\newtheorem{theorem}{Theorem}[section]
\newaliascnt{lemma}{theorem}
\newtheorem{lemma}[lemma]{Lemma}
\aliascntresetthe{lemma}
\newaliascnt{proposition}{theorem}
\newtheorem{proposition}[proposition]{Proposition}
\aliascntresetthe{proposition}
\newaliascnt{problem}{theorem}
\newtheorem{problem}[problem]{Problem}
\aliascntresetthe{problem}
\crefname{problem}{Problem}{Problems}
\newaliascnt{corollary}{theorem}

\aliascntresetthe{corollary}
\newaliascnt{claim}{theorem}
\newtheorem{claim}[claim]{Claim}
\aliascntresetthe{claim}
\newaliascnt{conjecture}{theorem}

\aliascntresetthe{conjecture}
\theoremstyle{definition}
\newaliascnt{definition}{theorem}
\newtheorem{definition}[definition]{Definition}
\aliascntresetthe{definition}
\theoremstyle{plain}
\newaliascnt{question}{theorem}
\newtheorem{question}[question]{Question}
\aliascntresetthe{question}

\crefname{conjecture}{Conjecture}{Conjectures}
\crefname{theorem}{Theorem}{Theorems}
\crefname{lemma}{Lemma}{Lemmas}
\crefname{proposition}{Proposition}{Propositions}
\crefname{corollary}{Corollary}{Corollaries}
\crefname{claim}{Claim}{Claims}
\crefname{definition}{Definition}{Definitions}
\crefname{question}{Question}{Questions}
\crefname{section}{Section}{Sections}
\Crefname{section}{Section}{Sections}
\crefname{subsection}{Subsection}{Subsections}
\Crefname{subsection}{Subsection}{Subsections}
\crefname{subsubsection}{Subsection}{Subsections}
\Crefname{subsubsection}{Subsection}{Subsections}

\newcommand{\ceil}[1]{\left\lceil #1\right\rceil}
\newcommand{\floor}[1]{\left\lfloor #1\right\rfloor}
\newcommand{\disc}{\operatorname{disc}}
\newcommand{\Reg}{\mathcal R}
\newcommand{\1}{\mathbf 1}
\newcommand{\oldeg}{\overline d}
\newcommand{\ole}{\overline e}

\title{Spanning $H$-subdivisions with Prescribed Path Lengths}
\author{
Zhilan Wang\thanks{School of Mathematics, Shandong University, Jinan 250100, China. Supported by the National Natural Science Foundation of China (No.~12571373).}
\and
Shuo Wei\footnotemark[1] \thanks{Corresponding author. Email: \href{mailto:shuowei@mail.sdu.edu.cn}{\texttt{shuowei@mail.sdu.edu.cn}}}
}
\date{}

\begin{document}
\maketitle
\vspace{-1.2em}
\begin{abstract}
\noindent
We study spanning $H$-subdivisions in dense graphs where the length of
every subdivision path is prescribed in advance. This problem is
motivated in part by a question of Pavez-Sign\'e
[\emph{Combin. Probab. Comput.} 33 (2024), 121--128], who asked
whether the subdivision paths in a spanning $H$-subdivision can be
required to have similar lengths.

Let $h\ge3$ be an integer and let $0<\beta\ll\alpha\ll1/h$. We prove that,
for all sufficiently large $n$, every $n$-vertex graph $G$ with
$\delta(G)\ge n/2+\floor{h/3}$
has the following property. For every graph $H$ with $h$ edges and no isolated vertices, write $E(H)=\{e_1,\ldots,e_h\}$, and every choice of integers
$\ell_1,\ldots,\ell_h\ge4$ satisfying
$\sum_{i=1}^h\ell_i=n-|V(H)|+h$ and
$\sum_{\ell_i<\alpha n}\ell_i\le\beta n$,
the graph $G$ contains a spanning $H$-subdivision in which the $i$th
edge of $H$ is replaced by a path of length exactly $\ell_i$. We also give a family of examples showing that a linear additive term in $h$ is necessary in general.
\end{abstract}



\section{Introduction}\label{sc1}

All graphs are finite, simple and undirected. For a positive integer
$t$, write $[t]=\{1,\ldots,t\}$. Let $H$ be a graph with
$E(H)=\{e_1,\ldots,e_h\}$. An \emph{$H$-subdivision} in a graph $G$
is specified by an injective map $\varphi\colon V(H)\to V(G)$ and
paths $P_1,\ldots,P_h$ such that, if $e_i=xy$, then $P_i$ joins
$\varphi(x)$ to $\varphi(y)$. The paths are pairwise internally
vertex-disjoint, and their internal vertices avoid $\varphi(V(H))$.
The vertices in $\varphi(V(H))$ are the \emph{branch vertices}, and
the paths $P_i$ are the \emph{subdivision paths}. The subdivision is
\emph{spanning} if its vertex set is $V(G)$. The length of a path is
its number of edges.

If $P_i$ has length $\ell_i$, then a spanning $H$-subdivision in an
$n$-vertex graph necessarily satisfies
$\sum_{i=1}^h\ell_i=n-|V(H)|+h$. We ask when the individual lengths $\ell_1,\ldots,\ell_h$ can be fixed
in advance.

Length control in subdivisions has been studied extensively in the
balanced setting. A subdivision is \emph{balanced} if all subdivision
paths have the same length. Thomassen~\cite{thomassen1984} conjectured
that sufficiently large average degree forces a balanced subdivision
of every fixed clique, and Liu and Montgomery~\cite{liu-montgomery}
resolved this conjecture. Wang~\cite{wang2023} obtained quantitative
bounds for this problem, while Luan, Tang, Wang and
Yang~\cite{luan2023} and Gil Fern\'andez, Hyde, Liu, Pikhurko and
Wu~\cite{gilfernandez2023} independently determined the correct
quadratic order of the required average degree. More recently, Kim
et al.~\cite{kim-length-2026} proved that a linear-in-$e(H)$ average
degree bound suffices for balanced subdivisions of arbitrary graphs
$H$.

For spanning subdivisions, Pavez-Sign\'e~\cite{Pavez} asked whether
similar control over the subdivision path lengths is possible.

\begin{question}[Pavez-Sign\'e~\cite{Pavez}]
\label{ques:pavez}
For every $\varepsilon>0$, does there exist a constant $C_0>0$ such
that, for every $C\ge C_0$ and every positive integer $h$, every graph
$G$ on $n=Ch$ vertices with
$\delta(G)\ge(1+\varepsilon)n/2$
contains, for every $h$-edge graph $H$ with no isolated vertices, a
spanning $H$-subdivision whose subdivision paths have similar lengths?
\end{question}

Lee~\cite{Lee} noted that his proof can be slightly modified to
produce an almost balanced spanning $H$-subdivision. Related nearly balanced spanning clique subdivision results in pseudorandom graphs were subsequently
obtained by Pavez-Sign\'e, Lee and Petrov~\cite{lee-pavez-petrov}.

A stronger form of length control is to prescribe the length of each
subdivision path in advance. Such questions also arise in graph
linkage, where the branch vertices are prescribed as well. Chizmar,
Magnant and Salehi Nowbandegani~\cite{chizmar2016} studied linkage
with almost prescribed path lengths. Coll, Magnant and Salehi
Nowbandegani~\cite{coll2019} subsequently proved a sharp Ore-type
theorem for spanning $H$-subdivisions with prescribed branch vertices
and path lengths. In the directed setting, Cheng, Wang and
Yan~\cite{cheng-wang-yan-linked} proved that, for every fixed
$h$-arc digraph $H$ without isolated vertices, every sufficiently
large digraph $D$ with $\delta^0(D)\ge n/2+h$ is arbitrary
Hamiltonian $H$-linked, in which both the branch vertices and
the subdivision path lengths are prescribed.

We show that the length of every subdivision path can be prescribed
in advance, while the branch vertices need not be specified. Our main
result is the following.

\begin{theorem}\label{thm:main}
Let $h\ge3$ be an integer and let $0<\beta\ll\alpha\ll1/h$.  Then
there exists $n_0$ such that the following holds for every $n\ge n_0$.
Let $H$ be a graph with $h$ edges and no isolated vertices, and write
$E(H)=\{e_1,\ldots,e_h\}$.  Let $G$ be an $n$-vertex graph with
$\delta(G)\ge n/2+\floor{h/3}$, and let
$\ell_1,\ldots,\ell_h\ge4$ be integers satisfying
$\sum_{i=1}^h\ell_i=n-|V(H)|+h$ and
$\sum_{\ell_i<\alpha n}\ell_i\le\beta n$.
Then $G$ contains a spanning $H$-subdivision in which, for every
$i\in[h]$, the edge $e_i$ is replaced by a path of length exactly
$\ell_i$.
\end{theorem}

For every fixed graph $H$,  \cref{thm:main} gives a stronger form of the length control asked for in \cref{ques:pavez}: subject to the stated condition on short paths, the length of every subdivision path may be prescribed in advance. In particular, the prescribed lengths of our result can be chosen to differ by at most one.

The minimum degree condition in \cref{thm:main} also connects with a
classical problem on spanning subdivisions near the Dirac threshold.
A spanning $C_3$-subdivision is precisely a Hamilton cycle, so
Dirac's theorem~\cite{dirac1952} is the basic example of a minimum
degree theorem for spanning subdivisions. Motivated by this viewpoint,
Babu and Diwan~\cite{babu-diwan} asked whether essentially the same
degree threshold suffices for every fixed graph.

\begin{problem}[Babu and Diwan~\cite{babu-diwan}]
\label{prob:babu-diwan}
For every fixed graph $H$ with no isolated vertices, does there exist
an integer $f(H)$ such that every graph $G$ of order $n\ge f(H)$ with
minimum degree at least $(n+1)/2$ contains a spanning
$H$-subdivision?
\end{problem}

The existence problem has also been studied in a uniform setting in
which $h$ is allowed to grow with $n$. Pavez-Sign\'e~\cite{Pavez} later proved an asymptotic Dirac-type theorem for sufficiently dense regular targets and conjectured that, for every $\varepsilon>0$ and all sufficiently large $C$, every graph $G$ on $n=Ch$ vertices with
$\delta(G)\ge(1/2+\varepsilon)n$
contains a spanning subdivision of every $h$-edge graph without
isolated vertices. Lee~\cite{Lee} proved this conjecture in the
stronger setting of digraphs. Wang, Cheng and
Yan~\cite{wang-cheng-yan2026} subsequently obtained the sharp
semidegree condition $\delta^0(D)\ge(n+h)/2-1$
for spanning subdivisions of $h$-arc digraphs when $n$ is sufficiently
large compared with $h$.

For every fixed graph $H$, \cref{thm:main} gives a minimum degree
condition only an additive constant above the Dirac threshold, and
hence provides a partial answer to \cref{prob:babu-diwan}. Moreover,
the conclusion is stronger in that, subject to the stated condition on
short paths, the length of every subdivision path may be prescribed
in advance.

Finally, in \cref{sec:lower-bounds}, we give a family of examples
showing that the linear dependence on $h$ in the minimum degree
condition is necessary in general when all subdivision lengths are
prescribed in advance.

\medskip
\noindent\textbf{Organization.}
In \cref{sc2}, we introduce the notation, define the two extremal
cases and give the lower bound construction. The non-extremal case is treated by absorption in \cref{sc3}. In \cref{sc4}, we establish the auxiliary tools for the extremal cases and complete the proof of \cref{thm:main}.  Finally we conclude with two open
questions in \cref{sc5}.

\section{Preliminaries}\label{sc2}

\subsection{Notation}
All graphs are finite, simple and undirected. For a positive integer $t$,
write $[t]=\{1,\ldots,t\}$. For a graph $G$, let
$|G|=|V(G)|$ and $e(G)=|E(G)|$. If $X\subseteq V(G)$, then $G[X]$
denotes the subgraph induced by $X$, and we write $e(X)=e(G[X])$.
For disjoint sets $X,Y\subseteq V(G)$, let $G[X,Y]$ be the bipartite
subgraph consisting of the edges between $X$ and $Y$, and write
$e(X,Y)=e(G[X,Y])$. Finally, for $v\in V(G)$ and
$X\subseteq V(G)$, write $N_G(v,X)=N_G(v)\cap X$ and
$d_G(v,X)=|N_G(v,X)|$.

We omit the subscript when the ambient graph is clear. For disjoint
$X,Y\subseteq V(G)$, set $\oldeg_Y(x)=|Y\setminus N(x)|$, $x\in X$ and  $\ole(X,Y)=|X||Y|-e(X,Y)$.

The order of a path is its number of vertices, and its length is its
number of edges. Unless stated otherwise, disjoint subgraphs are
vertex-disjoint.

If $L$ is bipartite, define
$\disc(L):=\min\{\bigl||P|-|Q|\bigr|:(P,Q)\text{ is a bipartition of }L\}$.
When $L$ is disconnected, the two bipartition classes of each component
may be interchanged independently. For an arbitrary graph $K$, let $\operatorname{bal}(K)$ denote the minimum number of vertices whose deletion leaves a bipartite graph with
two equally sized bipartition classes. We use the standard hierarchy notation for constants.

\subsection{The extremal cases}
For $\varepsilon>0$, we call a graph $G$
\emph{$\varepsilon$-non-extremal} if neither of the following conditions
holds.

\smallskip

\noindent\textbf{Extremal Case 1.}
There is a set $A\subset V(G)$ such that
$|A|\ge n/2-\varepsilon n$ and $e(A)\le\varepsilon n^2$.

\smallskip

\noindent\textbf{Extremal Case 2.}
There is a set $A\subset V(G)$ with
$n/2-\varepsilon n\le |A|\le n/2$ such that, for
$B=V(G)\setminus A$, we have $e(A,B)\le\varepsilon n^2$.

\begin{figure}[htbp]
\centering
\begin{tikzpicture}[
    line/.style={draw=black, line width=0.9pt, line join=round},
    lab/.style={font=\large\itshape},
    sublab/.style={font=\normalsize}
]

\def\xrad{0.85cm}
\def\yrad{1.55cm}

\begin{scope}
    \draw[line, fill=white] (0,0) ellipse [x radius=\xrad, y radius=\yrad];
    \draw[line, fill=white] (3.7,0) ellipse [x radius=\xrad, y radius=\yrad];

    \draw[line, fill=gray!45]
        (1.00,0) -- (1.42,0.60) -- (1.42,0.28)
        -- (2.28,0.28) -- (2.28,0.60)
        -- (2.70,0) -- (2.28,-0.60)
        -- (2.28,-0.28) -- (1.42,-0.28)
        -- (1.42,-0.60) -- cycle;

    \node[lab] at (0,-2.0) {$A$};
    \node[lab] at (3.7,-2.0) {$B$};

    \node[sublab] at (1.85,-2.75) {Extremal Case $1$};
\end{scope}

\begin{scope}[xshift=7.8cm]
    \draw[line, fill=gray!45] (0,0) ellipse [x radius=\xrad, y radius=\yrad];
    \draw[line, fill=gray!45] (3,0) ellipse [x radius=\xrad, y radius=\yrad];

    \node[lab] at (0,-2.0) {$A$};
    \node[lab] at (3,-2.0) {$B$};

    \node[sublab] at (1.5,-2.75) {Extremal Case $2$};
\end{scope}

\end{tikzpicture}
\caption{Two extremal cases.}
\label{figs}
\end{figure}
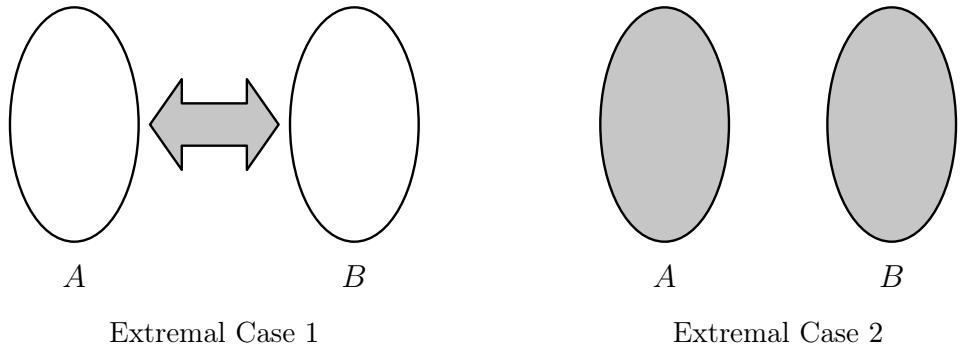

\subsection{Lower bound construction}\label{sec:lower-bounds}

We give a construction showing that the additive term in the degree
condition must in general be linear in $h$.

\begin{proposition}\label{prop:linear-lower}
For every integer $c\ge1$, let $h=6c-2$ and
$H_c=(2c-1)K_3\cup K_2$. Then, for infinitely many $n$, there are
admissible prescribed lengths and an $n$-vertex graph $G$ with
$\delta(G)=\ceil{n/2}+c-1$
such that $G$ contains no spanning $H_c$-subdivision with those
prescribed lengths.
\end{proposition}

\begin{proof}
Let $n=2m$ be sufficiently large. Since $e(H_c)=h=6c-2$ and $|V(H_c)|=6c-1$, the prescribed lengths must sum to $L=n-|V(H_c)|+h=n-1$. Choose one edge from each of the $2c-1$ triangles of $H_c$. Set
$$
q=\floor{\frac{L-(2c-1)}{2h}}
\quad\text{and}\quad
R=L-2qh-(2c-1).
$$
Then $R$ is even and $0\le R<2h$. Initially assign length $2q+1$
to the chosen edge of each triangle and length $2q$ to every other
edge of $H_c$. Finally, add $2$ to the lengths of any $R/2$ edges.
The resulting lengths sum to $L$, and the three lengths on each
triangle have odd sum. Moreover, for sufficiently large $n$, all
prescribed lengths are at least $\alpha n$.

Let $|A|=m-c+1$ and $|B|=m+c-1$, and choose
$S\subseteq B$ with $|S|=2c-2$. Define $G$ by taking all edges
between $A$ and $B$, no edges inside $A$, and precisely those edges
of $G[B]$ having at least one endpoint in $S$. Then
$d_G(a)=m+c-1$ for every $a\in A$,
$d_G(b)=m+c-1$ for every $b\in B\setminus S$, and every vertex of
$S$ has degree $n-1$. Hence
$\delta(G)=m+c-1=\ceil{n/2}+c-1$.

Suppose that $G$ contains a spanning $H_c$-subdivision with the
prescribed lengths. Each triangle of $H_c$ is then replaced by an odd
cycle. Every cycle contains an even number of edges between $A$ and
$B$, while $A$ is independent. Hence each of these $2c-1$
vertex-disjoint odd cycles contains an edge of $G[B]$. Choosing one
such edge from each cycle gives a matching of size $2c-1$ in $G[B]$,
contradicting the fact that $S$ is a vertex cover of $G[B]$ of order
$2c-2$.
\end{proof}

Since $h=6c-2$, the construction shows that no additive term $o(h)$
can guarantee spanning $H$-subdivisions with all subdivision path
lengths prescribed in advance, uniformly over all $h$-edge graphs $H$.


\section{The non-extremal case}\label{sc3}
Choose constants such that
$$
0<\frac1{n_0}\ll\xi\ll\beta\ll\alpha_0\ll\gamma
\ll\eta_0^2\ll\eta_0\ll\varepsilon\ll\alpha^2\ll\frac1h.
$$
Throughout this section, let $G$ be an $n$-vertex $\varepsilon$-non-extremal graph with $\delta(G)\ge n/2$, and let $H$ be a graph with $h\ge3$ edges and no isolated vertices.

\subsection{Path-covering lemma}\label{subsec:covering}
 The following lemma ensures that every sufficiently large non-extremal
graph with minimum degree close to half of its order contains a Hamilton
cycle, and hence a Hamilton path. This will be used to cover all
vertices left outside the absorbing structure.
\begin{lemma}[\cite{Khan}, Path-covering Lemma]\label{Almost-covering}
For any constant $\eta_0$ with $0<\eta_0\ll\varepsilon$ there exists a constant $t_0$ such that if $T$ is an $\varepsilon$-non-extremal graph on $t\geq t_0$ vertices with $\delta(T)\geq(1/2-\eta_0)t$, then $T$ is Hamiltonian.
\end{lemma}

\subsection{Absorbing lemma}\label{subsec:absorbing}
We construct a small $H$-subdivision whose long subdivision paths contain
many local absorbers. We first find many absorbers for every ordered pair,
then select a small disjoint family and distribute it among the long
subdivision paths. The resulting structure will absorb consecutive segments
of a Hamilton path in the final step.
\begin{lemma}\label{lem:prob}
The following standard estimates hold.
\begin{enumerate}[label=\textnormal{(\arabic*)}, ref=\thelemma(\arabic*)]
    \item \label[lemma]{lem321} If $X$ is a binomial random variable and
    $0<a<3/2$, then
    \[
    \mathbb P\bigl(|X-\mathbb EX|>a\,\mathbb EX\bigr)
    <2e^{-\frac{a^2}{3}\mathbb EX}.
    \]
    \item \label[lemma]{lem322} If $Y$ is a non-negative random variable and
    $b>0$, then
    \[
    \mathbb P(Y\ge b)\le \frac{\mathbb EY}{b}.
    \]
\end{enumerate}
\end{lemma}
Before introducing the absorbing structures, we set up a convenient notation for concatenating paths. If two paths $P=x\cdots y$ and $Q=y\cdots z$ satisfy
$V(P)\cap V(Q)=\{y\}$, then we write $P\circ Q$ for their concatenation.

\begin{definition}\label{def:absorber}

Let $u,v\in V(G)$ be an ordered pair of vertices, not necessarily distinct. An ordered $4$-tuple $z_1z_2z_3z_4$ in $G$ is called an \emph{absorber} for $(u,v)$ if $z_1z_2z_3z_4$ forms a path, $uz_2,\ vz_3\in E(G)$ and $\{z_1,z_2,z_3,z_4\}\cap\{u,v\}=\emptyset$. In this case, we also say that $z_1z_2z_3z_4$ \emph{absorbs} $(u,v)$. For each ordered pair $u,v\in V(G)$, let $\mathcal A_{uv}$ denote the family of all absorbers for $(u,v)$.
\end{definition}

In order for the probabilistic absorbing method to succeed, the next lemma guarantees that the graph $G$ contains a rich abundance of absorbers for every pair of vertices.

\begin{lemma}\label{lem:many-absorbers}
Let $G$ be an $n$-vertex graph with $\delta(G)\ge (1/2-\eta_0)n$ which is
$\varepsilon$-non-extremal.
Then for every ordered pair of vertices $u,v\in V(G)$, we have $|\mathcal A_{uv}|\ge \gamma n^4$.
\end{lemma}
\begin{proof}
Fix a ordered pair $u,v\in V(G)$, and let $X:=N(u)$ and $Y:=N(v)$.
Write $A:=X\setminus Y$, $B:=Y\setminus X$, $C:=X\cap Y$, and
$R:=V(G)\setminus (X\cup Y)$. Since $\delta(G)\ge (1/2-\eta_0)n$, we have
$|X|,|Y|\ge (1/2-\eta_0)n$, and hence $|R|=n-|X\cup Y|\le |C|+2\eta_0 n$.
We shall show that there are at least $cn^2$ edges with one end in $X$ and
the other in $Y$, where $c>0$ depends only on $\varepsilon$ and $\eta_0$.
We divide into three cases.

If $|C|\ge (1/2-\varepsilon)n$, since $G$ is $\varepsilon$-non-extremal, $G$ does not satisfy Extremal Case 1, so we must have $e(C)>\varepsilon n^2$.
Every edge of $G[C]$ has both ends in $X\cap Y$, thus there are at least
$\varepsilon n^2$ such middle edges.

Next assume that $4\eta_0 n\le |C|< (1/2-\varepsilon)n$. For every
$x\in C$, we have
$d(x,X\cup Y)\ge d(x)-|R|\ge (1/2-\eta_0)n-(|C|+2\eta_0 n)
=(1/2-3\eta_0)n-|C|$.
Since $|C|<(1/2-\varepsilon)n$ and $\eta_0\ll\varepsilon$, it follows that
$d(x,X\cup Y)\ge \varepsilon n/2$. Summing over all $x\in C$, we obtain at
least $|C|\varepsilon n/2$ incidences between $C$ and $X\cup Y$, so there are
at least $|C|\varepsilon n/4\ge \varepsilon\eta_0 n^2$ edges with one end in
$C$ and the other in $X\cup Y$. Every such edge can be ordered as an edge $z_2z_3$ with $z_2\in X$ and $z_3\in Y$, so again we have at least $\varepsilon\eta_0 n^2$ middle edges.

Finally, assume that $|C|<4\eta_0 n$. Then
$|A|=|X|-|C|\ge (1/2-\eta_0)n-4\eta_0 n=(1/2-5\eta_0)n$, and similarly
$|B|\ge (1/2-5\eta_0)n$. Let $S$ be the smaller of $A$ and $B$, and let
$T$ be the other one. Then $(1/2-5\eta_0)n\le |S|\le n/2$. Since $\eta_0\ll\varepsilon$ and $G$ is $\varepsilon$-non-extremal, $G$ does not
satisfy Extremal Case 2, and thus $e(S,V(G)\setminus S)>\varepsilon n^2$.
Moreover, $|C|+|R|\le 4\eta_0 n+(|C|+2\eta_0 n)<10\eta_0 n$, so $e(S,C\cup R)\leq 10\eta_0 n^2$. Hence
$e(S,T)\ge e(S,V(G)\setminus S)-e(S,C\cup R)\ge \varepsilon n^2-10\eta_0 n^2
\ge \varepsilon n^2/2$. Every edge between $S$ and $T$ is again an
edge with one end in $X$ and the other in $Y$.

Therefore, in all cases, there are at least $c n^2$ edges $z_2z_3$ with
$z_2\in X$ and $z_3\in Y$, where we may take $c:=\varepsilon\eta_0/2$.  Discard
the at most $2n$ such edges incident with $u$ or $v$, at least $c n^2/2$ edges remain.
For each remaining edge $z_2z_3$, choose $z_1\in N(z_2)\setminus\{u,v,z_3\}$ and then $z_4\in N(z_3)\setminus\{u,v,z_1,z_2\}$.
Since $\delta(G)\ge (1/2-\eta_0)n$, the number of choices for $(z_1,z_4)$ is at
least
$$
\bigl((1/2-\eta_0)n-3\bigr)\bigl((1/2-\eta_0)n-4\bigr)\ge \frac{n^2}{10}.
$$
Each resulting path $z_1z_2z_3z_4$ is an absorber
for $(u,v)$, hence $|\mathcal A_{uv}|\ge ((\varepsilon\eta_0/2)n^2/2)\cdot n^2/10\ge \gamma n^4$.
\end{proof}

The next lemma shows that one can select a small pairwise vertex-disjoint family of absorbers such that every pair still has many absorbers inside this family.

\begin{lemma}\label{lem:absorbing-family}
Let $G$ be an $n$-vertex graph with $\delta(G)\ge(1/2-\eta_0)n$ which is
$\varepsilon$-non-extremal.
Then there exists a family $\mathcal F$ of at most $\xi n$
pairwise vertex-disjoint absorbers in $G$ such that for every ordered pair $u,v\in V(G)$, $|\mathcal A_{uv}\cap \mathcal F|\ge \xi^2 n$.
\end{lemma}

\begin{proof}
Set $p=\frac{\xi}{1000}n^{-3}$, and let $\mathcal F'$ be a random family of ordered $4$-tuples of distinct vertices, obtained by including each ordered $4$-tuple independently with probability $p$. First, for sufficiently large $n$,
\[
 \frac{\xi n}{2000}\le \mathbb E|\mathcal F'|
 =p n(n-1)(n-2)(n-3)\le \frac{\xi n}{1000}.
\]
Since $2\mathbb E|\mathcal F'|\le \xi n$, Chernoff's inequality in \cref{lem321} gives
\begin{equation}\label{eq1}
 \mathbb P(|\mathcal F'|\ge \xi n)
 \le \mathbb P(|\mathcal F'|\ge2\mathbb E|\mathcal F'|)
 \le2e^{-\mathbb E|\mathcal F'|/3}=o(1).
\end{equation}
Next, fix a ordered pair $u,v\in V(G)$  and let $X_{uv}:=|\mathcal A_{uv}\cap \mathcal F'|$. By \cref{lem:many-absorbers}, we have $|\mathcal A_{uv}|\ge \gamma n^4$.
Thus $$\mathbb E X_{uv}=p|\mathcal A_{uv}|\ge \frac{\xi}{1000}n^{-3}\cdot \gamma n^4=\frac{\gamma\xi}{1000}n.$$
Since $\xi\ll\gamma$, we may choose the constants so that
$\gamma\xi/1000\ge 8\xi^2$.  Hence $\mathbb E X_{uv}\ge 8\xi^2 n$.
Another application of \cref{lem321} gives
$\mathbb P(X_{uv}\le \mathbb E X_{uv}/2)\le 2e^{-\mathbb E X_{uv}/12}$, and therefore $\mathbb P(X_{uv}\le 4\xi^2 n)\le 2e^{-2\xi^2 n/3}$. Since there are at most $n^2$ ordered pairs $(u,v)$, the union bound implies
that, with probability $1-o(1)$,
\begin{equation}\label{eq2}
    |\mathcal A_{uv}\cap \mathcal F'|\ge 4\xi^2 n
\end{equation}
for every pair $u,v\in V(G)$.

Let $Y$ be the number of unordered  intersecting pairs of ordered $4$-tuples in
$\mathcal F'$.  A given ordered $4$-tuple intersects at most $64n^3$ other ordered $4$-tuples. Hence $\mathbb EY\le \frac12 n^4\cdot64n^3\cdot p^2<\xi^2n/30000$.
By \cref{lem322}, 
\begin{equation}\label{eq3}
   \mathbb P(Y\ge \xi^2 n/2)\le \frac{\mathbb E Y}{\xi^2 n/2}<1/10000.
\end{equation}

Therefore, combining \eqref{eq1}, \eqref{eq2}, and \eqref{eq3} for all sufficiently large $n$, there exists a choice of $\mathcal F'$ such that the following three properties hold
simultaneously:
$$ |\mathcal F'| < \xi n,\quad  Y \le \frac{\xi^2 n}{2}, \quad \text{and} \quad |\mathcal A_{uv}\cap \mathcal F'| \ge 4\xi^2 n \text{ for every ordered pair } u,v.$$
Fix such a family $\mathcal F'$. From $\mathcal F'$ delete every ordered $4$-tuple that is not an absorber, and also delete every ordered $4$-tuple that intersects another member of $\mathcal F'$. Let $\mathcal F$ be the remaining family. Then $\mathcal F$ consists of pairwise vertex-disjoint absorbers and $|\mathcal F|\le |\mathcal F'|<\xi n$.

It remains to verify the absorption property. For a ordered pair $u,v\in V(G)$, we have $|\mathcal A_{uv}\cap\mathcal F'| \ge 4\xi^2 n$. Since each intersecting pair accounts for at most two such absorbers and $Y\le \xi^2 n/2$, at most $2Y\le \xi^2 n$ absorbers of $\mathcal A_{uv}\cap\mathcal F'$ are removed. Consequently, $|\mathcal A_{uv}\cap\mathcal F|\ge 4\xi^2 n-\xi^2 n> \xi^2 n$.
This proves the lemma.
\end{proof}

We next partition the family $\mathcal F$ into several groups and connect the
absorbers within each group into a path. 

\begin{lemma}\label{lem:partition-linking}
Let $G$ be a graph on $n$ vertices with $\delta(G)\ge (1/2-\eta_0)n$ which is $\varepsilon$-non-extremal, and let $\mathcal F$ be the family obtained in \cref{lem:absorbing-family}, and let $f:=|\mathcal F|$. Let $t\ge2$, and let $\tau$ be a real number with $\xi^2<\tau<1-\xi^2$. Suppose that $f=f_1+\cdots+f_t$ and $\tau f<f_i<(1-\tau)f$ for every $i\in[t]$. There exists a partition $\mathcal F=\mathcal F_1\cup\cdots\cup \mathcal F_t$ with $|\mathcal F_i|=f_i$ for every $i\in [t]$ such that
\begin{enumerate}[label=\textnormal{(\roman*)}, ref=(\roman*)]
    \item \label{lem361} for every ordered pair $u,v\in V(G)\setminus V(\mathcal F)$ and every $i\in [t]$, the family $\mathcal F_i$ contains at least one absorber for $(u,v)$;

    \item \label{lem362} writing $\mathcal F_i=\{F_{i,1},\dots,F_{i,f_i}\}$, there exists a path $L_i$ in $G$ of the form $L_i= F_{i,1}\circ P_{i,1}\circ F_{i,2}\circ \cdots
    \circ P_{i,f_i-1}\circ F_{i,f_i}$, where each $P_{i,j}$ has order at most $4$, and the paths $L_1,\dots,L_t$ are pairwise vertex-disjoint.
\end{enumerate}
\end{lemma}

\begin{proof}
 We first prove \ref{lem361}. Choose uniformly at random a partition $\mathcal F=\mathcal F_1\cup\cdots\cup\mathcal F_t$ with $|\mathcal F_i|=f_i$ for every $i\in[t]$. Fix a ordered pair $u,v\in V(G)\setminus V(\mathcal F)$ and an index $i\in [t]$, and let $Z_{uv}^i$ denote the number of absorbers for $(u,v)$ that lie in $\mathcal F_i$. Then we have
$$\mathbb E Z_{uv}^i=\frac{|\mathcal A_{uv}\cap \mathcal F|}{|\mathcal F|}\,f_i.$$
By \cref{lem:absorbing-family}, we have $|\mathcal A_{uv}\cap \mathcal F|\ge \xi^2 n$ and $|\mathcal F|\le \xi n$ and hence $\mathbb E Z_{uv}^i\ge \xi f_i$.  On the other hand, $f_i>\tau f=\tau |\mathcal F|$, where $\tau>\xi^2$, and moreover $f=|\mathcal F|\ge |\mathcal A_{uv}\cap \mathcal F|\ge \xi^2 n$. Therefore $\mathbb E Z_{uv}^i> \xi^5 n$. 

By \cref{lem321}, 
$$\mathbb P\!\left(Z_{uv}^i\le \frac{\mathbb E Z_{uv}^i}{2}\right)
\le 2e^{-\mathbb E Z_{uv}^i/12}
\le 2e^{-\xi^5 n/12}.$$
Since there are at most $n^2$ ordered pairs $(u,v)$ and $t<1/\tau<1/\xi^2$, 
the union bound implies that with positive probability, every $\mathcal F_i$ contains at least one absorber for every ordered pair $u,v\in V(G)\setminus V(\mathcal F)$ and every $i\in[t]$, this implies $Z_{uv}^i\ge1$.  Fix such a partition. This proves \ref{lem361}.

Before proving \ref{lem362}, we first show the next claim, which allows us to connect two distinct vertices after avoiding a set of vertices.

\begin{claim}\label{clm:connect}
Let $W\subseteq V(G)$ with $|W|\le \gamma n$, and let $x,y\in V(G)\setminus W$ be distinct vertices. Then there exists an $x,y$-path of order at most $4$ in $G-W$.
\end{claim}

\begin{proof}[Proof of \cref{clm:connect}]\renewcommand*{\qedsymbol}{$\blacksquare$}
Set $G':=G-W$. We may assume that $xy\notin E(G')$. Let $X:=N_{G'}(x)$ and $Y:=N_{G'}(y)$. Since
$\delta(G)\ge (1/2-\eta_0)n$, we have
$|X|,|Y|\ge (1/2-\eta_0)n-|W|\ge (1/2-\eta_0-\gamma)n\ge (1/2-\varepsilon) n$.

If $X\cap Y\neq\emptyset$, then $x$ and $y$ have a common neighbor in $G'$, so there is an $x,y$-path of order $3$ in $G'$. Thus we may assume that $X\cap Y=\emptyset$. If $e_{G'}(X,Y)>0$, then an edge between $X$ and $Y$ yields an $x,y$-path of order $4$ in $G'$. Hence it remains to consider the case $e_{G'}(X,Y)=0$.

Let $R:=V(G')\setminus (X\cup Y)$. Since $X$ and $Y$ are disjoint and each has size at least $(1/2-\eta_0-\gamma)n$, we have $|R|\le (2\eta_0+2\gamma)n<\varepsilon n/2$. Let $A$ be the smaller of $X$ and $Y$. Then $n/2-\varepsilon n\le |A|\le n/2$. Moreover, $e_{G'}(A,V(G')\setminus A)=e_{G'}(A,R)\le |A||R|<\varepsilon n^2/2$. Furthermore, the number of edges between $A$ and $W$ in $G$ is bounded by $e_G(A, W) \le |A||W| \le \gamma n^2 < \varepsilon n^2/2$. Consequently, we have
$$e_G(A,V(G)\setminus A) \le e_{G'}(A,V(G')\setminus A)+e_G(A, W)< \varepsilon n^2.$$ This contradicts the assumption that $G$ is $\varepsilon$-non-extremal. Therefore $e_{G'}(X,Y)>0$, and the claim follows.
\end{proof}
We now construct the paths $L_1,\dots,L_t$ one by one. Fix $i\in [t]$, and write
$\mathcal F_i=\{F_{i,1},\dots,F_{i,f_i}\}$, where we view each absorber $F_{i,j}=(z_1,z_2,z_3,z_4)$ as the path $z_1z_2z_3z_4$, with initial vertex $z_1$ and terminal vertex $z_4$. We connect these absorbers in the given order. For each $j\in [f_i-1]$, let $b_{i,j}$ be the terminal vertex of $F_{i,j}$ and let $a_{i,j+1}$ be the initial vertex of $F_{i,j+1}$.

Suppose that for some $j\in [f_i-1]$ we have already constructed the path
$S_{i,j}:=F_{i,1}\circ P_{i,1}\circ F_{i,2}\circ \cdots \circ P_{i,j-1}\circ F_{i,j}$, where the previously chosen connecting paths are pairwise internally disjoint and avoid all absorbers not yet used. Let $W$ consist of all vertices in
$V(\mathcal F)\setminus \{b_{i,j},a_{i,j+1}\}$ together with all internal vertices of the connecting paths already chosen for $L_1,\dots,L_{i-1}$ and for
$P_{i,1},\dots,P_{i,j-1}$. Since $f=|\mathcal F|\le \xi n$, we have $|V(\mathcal F)|=4f\le 4\xi n$. Moreover, the total number of internal vertices used by all previously chosen connecting paths is at most $2f\le 2\xi n$. Hence $|W|\le 6f\le 6\xi n<\gamma n$.

Applying \cref{clm:connect} to $G-W$ with $x=b_{i,j}$ and $y=a_{i,j+1}$, we obtain a $b_{i,j},a_{i,j+1}$-path $P_{i,j}$ of order at most $4$  avoiding all previously used vertices. Therefore $S_{i,j+1}:=S_{i,j}\circ P_{i,j}\circ F_{i,j+1}$ is again a path. Repeating this procedure, we iteratively construct the path 
$$L_i = F_{i,1} \circ P_{i,1} \circ F_{i,2} \circ \cdots \circ P_{i,f_i-1} \circ F_{i,f_i}.$$
By construction, the paths $L_1,\dots,L_t$ are pairwise vertex-disjoint. This completes the proof of \ref{lem362}, and hence of the lemma.
\end{proof}

We are now in a position to prove the absorbing lemma by using \cref{lem:absorbing-family,lem:partition-linking}.

\begin{lemma}[Absorbing Lemma]\label{Absorbing}
Let $H$ be a graph with vertex set $V(H)=\{v_1,\ldots,v_s\}$, with $h$
edges and no isolated vertices. Let $G$ be an $n$-vertex $\varepsilon$-non-extremal graph with $n\ge n_0$ and $\delta(G)\ge n/2$. Suppose that $\ell_1,\ldots,\ell_h\ge3$ and that $\sum_{\ell_i<\alpha_0n}\ell_i\le\beta n$. Then $G$ contains an $H$-subdivision $H'$ with $|V(H')|\le\gamma n$ such
that the path replacing $e_i$ has length at most $\ell_i$ for every
$i\in[h]$. Moreover, it has length exactly $\ell_i$ whenever
$\ell_i<\alpha_0n$, and every image path with $\ell_i\ge\alpha_0n$
contains an absorber for every ordered pair in $V(G)\setminus V(H')$.
\end{lemma}

We first record a short path claim which will be used to construct the short
image paths with their prescribed lengths.

\begin{claim}\label{clm:short-path}
Let $W\subseteq V(G)$ with $|W|\le 2\gamma n$, let
$x,y\in V(G)\setminus W$ be distinct vertices, and let $\ell$ be an integer
with $3\le \ell\le \alpha_0 n$.  Then there exists an $x,y$-path of length
exactly $\ell$ in $G-W$.
\end{claim}

\begin{proof}\renewcommand*{\qedsymbol}{$\blacksquare$}
Choose non-negative integers $r$ and $s$ with $r+s=\ell-3$. Either of the two paths below is allowed to consist only of its endpoint. We first
greedily choose two vertex-disjoint paths
$x=x_0x_1\cdots x_r$ and $y_s y_{s-1}\cdots y_0=y$ in $G-W$.  Indeed, at
each step fewer than $|W|+\ell+2\le 2\gamma n+\alpha_0 n+2$ vertices are
forbidden, while $\delta(G)\ge n/2$ and $2\gamma+\alpha_0\ll1$, hence the
current endpoint always has an available neighbor.

Define $a:=x_r$ and $b:=y_s$, and let $U:=W\cup\{x_0,\ldots,x_r,y_0,\ldots,y_s\}$.  Define
$X:=N_G(a)\setminus U$ and $Y:=N_G(b)\setminus U$.  Then $|X|,|Y|\ge n/2-|U|\ge (1/2-2\gamma-\alpha_0)n-2\ge (1/2-\varepsilon)n$.

We claim that there is an edge with one endpoint in $X$ and the other in $Y$.
Suppose not.  Write $A:=X\setminus Y$, $B:=Y\setminus X$,
$C:=X\cap Y$, and $R:=V(G)\setminus (X\cup Y)$.  Since
$|X|,|Y|\ge (1/2-2\gamma-\alpha_0)n-2$, we have$ |R|\le |C|+(4\gamma+2\alpha_0)n+4$.
If $|C|\ge (1/2-\varepsilon)n$, then $e(C)=0$, since there is no edge with
one endpoint in $X$ and the other in $Y$.  Hence $G$ satisfies Extremal
Case~$1$, a contradiction.

Next suppose that $4(2\gamma+\alpha_0)n\le |C|<(1/2-\varepsilon)n$.  Since
there is no edge with one endpoint in $X$ and the other in $Y$, every vertex
of $C$ has no neighbor in $X\cup Y$.  Thus every $c\in C$ satisfies
$d(c)\le |R|\le |C|+(4\gamma+2\alpha_0)n+4<n/2$, where the last inequality
follows from $\alpha_0,\gamma\ll\varepsilon$ and $n$ sufficiently large.  This
contradicts $\delta(G)\ge n/2$.

Finally suppose that $|C|<4(2\gamma+\alpha_0)n$.  Then
$|A|,|B|\ge (1/2-\varepsilon)n$.  Let $S$ be the smaller of $A$ and $B$.
Then $n/2-\varepsilon n\le |S|\le n/2$.  Moreover, since there is no edge
with one endpoint in $X$ and the other in $Y$, all edges from $S$ to
$V(G)\setminus S$ go into $R$.  As $|R|<\varepsilon n$, we obtain
$e(S,V(G)\setminus S)<\varepsilon n^2$.  Thus $G$ satisfies Extremal
Case~$2$, again a contradiction.

Therefore there exists an edge $zz'$ with $z\in X$ and $z'\in Y$.  Then
$x=x_0x_1\cdots x_r z z' y_s y_{s-1}\cdots y_0=y$ is an $x,y$-path in $G-W$ of length $r+1+1+1+s=r+s+3=\ell$.
\end{proof}

\begin{proof}[\textbf{Proof of \cref{Absorbing}}]
Reorder the edges of $H$ as $e_1,\ldots,e_h$ so that, for some
$t\in\{0,\ldots,h\}$, we have
$$\ell_1\ge\cdots\ge \ell_t\ge \alpha_0 n>\ell_{t+1}\ge\cdots\ge \ell_h.$$ 
Thus $e_1,\ldots,e_t$ are the `long' edges and $e_{t+1},\ldots,e_h$ are the `short'
edges.  Since $H$ has no isolated vertices, $|V(H)|\le 2h$.  For each
$i\in[h]$, write $e_i=x_i y_i$.

Assume first that $t\ge1$. By \cref{lem:absorbing-family}, there exists a family
$\mathcal F$ of pairwise vertex-disjoint absorbers such that
$|\mathcal F|\le \xi n$ and
$|\mathcal A_{uv}\cap\mathcal F|\ge \xi^2 n$ for every pair $(u,v)$ with
$u,v\in V(G)$.  Set $f:=|\mathcal F|$.  If $t=0$, let $\mathcal F:=\emptyset$ and $f:=0$.

We now prepare absorbing paths for the `long' edges.  If $t=1$, we connect all
absorbers in $\mathcal F$ into one path $L_1$, using \cref{clm:connect} as in
the proof of \cref{lem:partition-linking}.  Then $L_1$ contains an absorber
for every pair $(u,v)$ with $u,v\in V(G)\setminus V(\mathcal F)$, and
$\ell(L_1)\le 6f-3$.  If $t\ge2$, choose $\tau$ with
$\xi^2<\tau<1/(2h)$.  Since $t\le h$ and $f\ge \xi^2n$, for sufficiently
large $n$ we may choose integers $f_1,\ldots,f_t$ with
$f=f_1+\cdots+f_t$ and $\tau f<f_i<(1-\tau)f$ for every $i\in[t]$.
Applying \cref{lem:partition-linking}, we obtain pairwise vertex-disjoint
paths $L_1,\ldots,L_t$ such that each $L_i$ is obtained by connecting the
absorbers in $\mathcal F_i$.  Moreover, each $L_i$ contains an absorber for
every pair $(u,v)$ with $u,v\in V(G)\setminus V(\mathcal F)$, and
$\ell(L_i)\le 6f_i-3\le 6\xi n$.

If $t=0$, no path $L_i$ is needed.  In all cases, $\left|\bigcup_{i=1}^t V(L_i)\right|\le 6f$. Choose an injective map
$\psi:V(H)\to V(G)\setminus\bigcup_{i=1}^t V(L_i)$, and set
$B:=\psi(V(H))$.  This is possible since
$\left|\bigcup_{i=1}^t V(L_i)\right|\le 6\xi n$ and $|V(H)|\le 2h$.

We first construct the image paths for the `short' edges.  Suppose that
$Q_{t+1},\ldots,Q_{i-1}$ have already been constructed.  Let $W_i$ consist of
the vertices in $B\setminus\{\psi(x_i),\psi(y_i)\}$, all vertices in
$\bigcup_{j=1}^tV(L_j)$, and all internal vertices of the previously
constructed `short' paths.  Since the total length of the `short' edges is at
most $\beta n$, we have $|W_i|\le 2h+6\xi n+\beta n<2\gamma n$.  By
\cref{clm:short-path}, there exists a $\psi(x_i),\psi(y_i)$-path $Q_i$ of
length exactly $\ell_i$ in $G-W_i$.  Repeating this for
$i=t+1,\ldots,h$, we obtain pairwise internally disjoint image paths for all
`short' edges.

It remains to construct the image paths for the `long' edges.  For each
$i\in[t]$, write $L_i=a_i\cdots b_i$.  Suppose that $Q'_1,\ldots,Q'_{i-1}$ have already been constructed.  Let $U_i$ consist of the vertices in $B\setminus\{\psi(x_i)\}$, all vertices in
$\bigcup_{j=1}^tV(L_j)\setminus\{a_i\}$, all internal vertices of the `short'
paths, and all internal vertices of the already constructed `long' paths.  Then
$|U_i|\le 6f+\beta n+6h<\gamma n$.  By \cref{clm:connect}, there exists a $\psi(x_i),a_i$-path $R_i$ of order at most $4$ in $G-U_i$.

Next let $W_i^+$ consist of the vertices in $B\setminus\{\psi(y_i)\}$, all vertices in
$\bigcup_{j=1}^tV(L_j)\setminus\{b_i\}$, all internal vertices of the `short'
paths, all internal vertices of the already constructed `long' paths, and all
internal vertices of $R_i$.  Then $|W_i^+|<2\gamma n$.  By
\cref{clm:short-path} with $\ell=3$, there exists a $b_i,\psi(y_i)$-path
$S_i$ of length exactly $3$ in $G-W_i^+$.

Define $Q'_i:=R_i\circ L_i\circ S_i$.  By construction, $Q'_i$ is a $\psi(x_i),\psi(y_i)$-path whose internal vertices are disjoint from all
previously constructed image paths and from all branch vertices except its
own endpoints.  Since $R_i$ has order at most $4$, it has length at most $3$;
also $S_i$ has length exactly $3$.  Thus $\ell(Q'_i)\le 3+\ell(L_i)+3\le 6\xi n+6<\alpha_0 n\le \ell_i$.  Repeating this for all $i\in[t]$, we obtain pairwise internally disjoint image paths for all `long' edges.

Let $H'$ be the union of the branch vertices $B$, the `short' image paths
$Q_{t+1},\ldots,Q_h$, and the `long' image paths $Q'_1,\ldots,Q'_t$.  Then
$H'$ is an $H$-subdivision.  The image path of $e_i$ has length exactly
$\ell_i$ for every $i>t$, and length at most $\ell_i$ for every $i\le t$.

For each $i\in[t]$, the path $Q'_i$ contains $L_i$ as a subpath.  By the
construction of $L_i$, the path $L_i$ contains an absorber for every pair
$(u,v)$ with $u,v\in V(G)\setminus V(\mathcal F)$.  Since
$V(G)\setminus V(H')\subseteq V(G)\setminus V(\mathcal F)$, each `long' image path $Q'_i$ contains an absorber for every pair $(u,v)$ with $u,v\in V(G)\setminus V(H')$.

Finally,
$$
|V(H')|
\le |V(H)|+\sum_{i=t+1}^h(\ell_i-1)+6f+4t
\le 2h+\beta n+6\xi n+4h
<\gamma n.
$$
\end{proof}

\subsection{Completion of the non-extremal case}

Let $\ell_1,\ldots,\ell_h$ satisfy the assumptions of
\cref{thm:main}. Since $\alpha_0<\alpha$, the condition on the short
subdivision paths implies that those of length less than $\alpha_0n$
have total length at most $\beta n$. Thus the assumptions of
\cref{Absorbing} are satisfied.

Relabel the edges of $H$ so that
$\ell_1\ge\cdots\ge\ell_h$, and let $t$ be the number of indices $i$
for which $\ell_i\ge\alpha_0n$. We have $t\ge1$. Indeed, otherwise all
the prescribed lengths would be less than $\alpha_0n<\alpha n$, so the
condition on the short paths would give
$\sum_{i=1}^h\ell_i\le\beta n$, whereas the spanning length condition
gives $\sum_{i=1}^h\ell_i=n-v(H)+h$, a contradiction for large $n$.

By \cref{Absorbing}, $G$ contains an $H$-subdivision $H'$ with
$|V(H')|\le\gamma n$ such that the image of $e_i$ has length exactly
$\ell_i$ for every $i>t$ and length at most $\ell_i$ for every $i\le t$.
Moreover, for each $i\le t$, this image path has length less than
$\alpha_0n$ and contains an absorber for every ordered pair of vertices
in $V(G)\setminus V(H')$.

For each $i\le t$, let $Q_i$ be the image of $e_i$ in $H'$ and let
$q_i$ be its length. Then $q_i<\alpha_0n\le\ell_i$. Let
$G_0=G-V(H')$ and $n'=|G_0|$. Since $|V(H')|\le\gamma n$, we have
$n'\ge(1-\gamma)n$.

We claim that $G_0$ is $(\varepsilon/2)$-non-extremal. Suppose first that
some $A\subseteq V(G_0)$ witnesses Extremal Case~1 in $G_0$. Then
$|A|\ge(1/2-\varepsilon/2)n'\ge(1/2-\varepsilon)n$, while
$e_G(A)=e_{G_0}(A)<\varepsilon n^2$. Hence $A$ also witnesses Extremal
Case~1 in $G$, a contradiction.

Suppose instead that $A$ witnesses Extremal Case~2 in $G_0$. The same
size estimates show that $A$ satisfies the required size condition in
$G$, and
\[
e_G\bigl(A,V(G)\setminus A\bigr)
 \le e_{G_0}\bigl(A,V(G_0)\setminus A\bigr)
    +|A||V(H')|
 <\varepsilon n^2,
\]
again a contradiction. Thus $G_0$ is $(\varepsilon/2)$-non-extremal.

Moreover,
$\delta(G_0)\ge\delta(G)-|V(H')|
\ge n/2-\gamma n\ge(1/2-\eta_0)n'$.
Therefore \cref{Almost-covering}, applied with $\varepsilon/2$ in place
of $\varepsilon$, yields a Hamilton path $P$ in $G_0$.

Using the spanning length condition and the lengths of the image paths
in $H'$, we obtain
\[
n' =n-|V(H')| =n-v(H)+h-\sum_{i=1}^t q_i-\sum_{i=t+1}^h\ell_i
 =\sum_{i=1}^t(\ell_i-q_i).
\]
We may therefore split $P$ into consecutive subpaths
$P_1,\ldots,P_t$, where $P_i$ has order $\ell_i-q_i$. Write
$P_i=c_i\cdots d_i$. Since $Q_i$ contains an absorber
$z_1z_2z_3z_4$ for $(c_i,d_i)$, replace this subpath of $Q_i$ by
$z_1z_2P_i z_3z_4$. This preserves the endpoints of $Q_i$ and increases
its length by $|V(P_i)|=\ell_i-q_i$.

After performing this replacement for every $i\le t$, the image of each
edge $e_i$ has length exactly $\ell_i$, and the resulting subdivision
covers all vertices of $G$. Hence it is the required spanning
$H$-subdivision. This completes the proof of the non-extremal case of \cref{thm:main}.

\section{The extremal cases}\label{sc4}
Throughout this section, $G$ satisfies the assumptions of
\cref{thm:main}. Choose $\eta>0$ such that
$$
\varepsilon\ll\eta^4\ll\eta\ll\alpha^2.
$$
Set $k:=\floor{h/3}$. We call an edge $e_i\in E(H)$ 
\emph{long} if $\ell_i\ge\alpha n$ and \emph{short}
otherwise.

\subsection{Auxiliary tools}\label{subsec:ext-tools}

We first establish a finite lemma that will provide the parity,
balance and corrections needed in the two extremal cases.

For a map $f:V(H)\to\{A,B\}$, write
$U_A=f^{-1}(A)$ and $U_B=f^{-1}(B)$, and set
$D_f=|U_A|+e_H(U_B)-|U_B|-e_H(U_A)$.
Given $p=(p_e)_{e\in E(H)}\in\{0,1\}^{E(H)}$, call an edge $e=uv$ a
\emph{defect} of $f$ if
$p_e\not\equiv\1\{f(u)\ne f(v)\}\pmod2$, and denote the number of
defects by $z_p(f)$.

\begin{lemma}\label{correction}
Let $H$ be a graph with $h\ge3$ edges and no isolated vertices, and set
$k=\floor{h/3}$.
\begin{enumerate}[label=\textup{(\roman*)},leftmargin=*]
\item If $K$ is obtained from $H$ by subdividing an arbitrary set of edges
exactly once, then
$\operatorname{bal}(K)\le\floor{2h/3}$.

\item For every $p\in\{0,1\}^{E(H)}$, there is a map
$f:V(H)\to\{A,B\}$ such that $z_p(f)+D_f\le2k$.

\item For every $e=uv\in E(H)$, there is an ordering
$\pi=(v_1,\ldots,v_{v(H)})$ of $V(H)$ such that, for every
$j\in[v(H)]$, writing $P_j=\{v_1,\ldots,v_j\}$, we have
$$
e_{H-e}\bigl(P_j,V(H)\setminus P_j\bigr)
 +2-\1\bigl\{|\{u,v\}\cap P_j|=1\bigr\}\le2k+2.
$$
\end{enumerate}
\end{lemma}

\begin{proof}
We prove the three assertions separately.

\smallskip
\noindent\emph{\textbf{Proof of \textup{(i)}.}}
Let $X=\{v\in V(H):d_H(v)\ge3\}$, put $s=|X|$, and let $g$ be the number
of edges lying in cycle components of $H$. Every component of $K-X$ is a
path or a cycle, and every cycle among them arises from a cycle component
of $H$.

Delete $X$ and one vertex from each odd cycle of $K-X$, and let $o$ be
the number of vertices deleted from odd cycles. The remaining graph is
bipartite. Its path components of odd order are the only components with
nonzero discrepancy, and their signs may be chosen independently.
Consequently, $\operatorname{bal}(K)\le s+o+1$, where the final $1$ is
unnecessary if the number of odd-order path components is even.

Since no vertex of $X$ lies in a cycle component counted by $g$, we have
$3s\le2(h-g)$ and $o\le g/3$. Since every cycle component contains at
least three edges, either $g=0$ or $g\ge3$. If $g\ge3$, then
$s+o+1\le 2(h-g)/3+ g/3+1\le 2h/3$, proving the assertion.

We may therefore assume that $g=0$. Put $C=\floor{2h/3}$. The preceding
argument applies unless $s=C$ and $K-X$ has an odd number of odd-order
path components. We assume this exceptional case.

Let $L=V(H)\setminus X$ and $w=\sum_{u\in L}d_H(u)$. Since
$3s\le2h-w$, we have $w=0$, $w\le2$, or $w\le1$ according as
$h\equiv0,1$, or $2\pmod3$.

Let $H_0$ be the spanning subgraph of $H$ formed by the edges subdivided
in $K$, set $m=e(H_0[X])$, and let $a_L$ be the number of edges of $H_0$
with at least one endpoint in $L$. Since $K-X$ has odd order,
$|L|+m+a_L$ is odd. Set
$\varepsilon_L=(|L|+a_L)\bmod2$, so that $m+\varepsilon_L$ is odd.

We claim that there is a vertex $v\in X$ with $N_H(v,L)=\varnothing$ and
$2d_{H_0[X]}(v)\le m+\varepsilon_L+1$.
Let $Y=\{x\in X:N_H(x,L)=\varnothing\}$. At most $w$ vertices of $X$
meet $L$, so $|Y|\ge s-w$. If $|Y|\ge4$, then
$\sum_{x\in Y}d_{H_0[X]}(x)\le2m$, and some $v\in Y$ has
$d_{H_0[X]}(v)\le\floor{m/2}$.

Suppose that $|Y|\le3$. The bounds on $w$ leave only
$(h,s,w)=(3,2,0)$, $h=4$ with $s=2$ and $w\le2$, $h=5$ with $s=3$
and $w\le1$, or $h=7$ with $s=4$. The first three possibilities contradict
$3s\le2e_H(X)+e_H(X,L)\le s(s-1)+w$. Hence $h=7$ and $s=4$.

Let $q=e_H(X,L)\le2$. If $q=0$, then $Y=X$ and the averaging argument
applies. If $q=1$, then $H[X]=K_4$ and $|Y|=3$. Some $v\in Y$ satisfies
$d_{H_0[X]}(v)\le\ceil{m/2}$: for $m\le4$ this follows by summing the
degrees over $Y$, while for $m\ge5$ it follows from
$d_{H_0[X]}(v)\le3$. If $q=2$, then $H[X]=K_4-e$ and $|Y|\ge2$.
For any two vertices of $Y$, the sum of their $H_0[X]$-degrees is at most
$m+1$, so one of them again has degree at most $\ceil{m/2}$.
Since $m+\varepsilon_L$ is odd, this proves the claim.

Fix such a vertex $v$ and let $r=d_{H_0[X]}(v)$. After deleting
$X\setminus\{v\}$, the remaining graph is the disjoint union of a star
with center $v$ and $r$ leaves, $m-r$ isolated vertices, and a collection
of paths on $|L|+a_L$ vertices.

Let $u$ be the number of odd-order paths in this collection. Then
$u\equiv\varepsilon_L\pmod2$ and $u\ge\varepsilon_L$. The star has
discrepancy $|1-r|$, while there are $m-r+u$ other components of
discrepancy one. Their signs can be chosen to cancel the discrepancy of
the star: the required parity follows from $m+\varepsilon_L$ being odd,
and the required magnitude follows from
$2r\le m+\varepsilon_L+1$. Hence deleting only $s-1$ vertices leaves a
balanced bipartite graph. Therefore
$\operatorname{bal}(K)\le s-1<C$, proving \textup{(i)}.

\smallskip
\noindent\emph{\textbf{Proof of \textup{(ii)}.}}
color the vertices of $H$ independently and uniformly with $A$ and $B$.
Each edge is a defect with probability $1/2$, and the defect indicators
of any two distinct edges are independent. Hence
$\mathbb Ez_p=h/2$ and $\operatorname{Var}(z_p)=h/4>0$.

If $h$ is odd, some coloring has
$z_p\le\floor{h/2}$. If $h$ is even, the positive variance implies that
some coloring has $z_p\le h/2-1$. Thus in either case there is a
coloring with $z_p\le\floor{(h-1)/2}$.

Complementing all colors leaves $z_p$ unchanged and replaces $D_f$
by $-D_f$. Choosing the better coloring from this complementary pair
gives
$z_p(f)+D_f\le\floor{(h-1)/2}\le2\floor{h/3}=2k$.

\smallskip
\noindent\emph{\textbf{Proof of \textup{(iii)}.}}
We first show that every graph $J$ with $m\ge2$ edges has a vertex ordering
in which every prefix cut has size at most $\floor{2m/3}$. We argue by
induction on $m$.

If $\Delta(J)\le2$, order each path component along the path and each
cycle component cyclically, placing the components consecutively. Every
prefix cut then has size at most one when $m=2$ and at most two when
$m\ge3$.

Now let $x$ be a vertex of degree $d\ge3$ and let $r=m-d$. If $r\ge2$,
take an inductive ordering of $J-x$; if $r=1$, place the ends of its unique
edge consecutively; and if $r=0$, take any ordering. List the neighbors
of $x$ in the order in which they occur and insert $x$ immediately after
the $\floor{d/2}$th neighbor. The edges incident with $x$ then contribute
at most $\ceil{d/2}$ to every prefix cut.

If $r=0$ or $r\ge2$, then every prefix cut has size at most
$\floor{2r/3}+\ceil{d/2}\le\floor{2m/3}$, where we use
$\ceil{d/2}\le\floor{2d/3}$ and $m=r+d$.

If $r=1$ and $d\ge4$, then every prefix cut has size at most
$1+\ceil{d/2}\le\floor{2m/3}$, since
$\ceil{d/2}\le\floor{(2d-1)/3}$ and $m=d+1$.

It remains to consider $r=1$ and $d=3$. Write the edges incident with
$x$ as $xa,xb,xc$. If the remaining edge is disjoint from
$\{a,b,c\}$, order the relevant vertices as $a,b,x,c,y,z$; if it is
$ab$, use $a,b,x,c$; and if it is $ay$ with $y\notin\{a,b,c\}$, use
$y,a,x,b,c$. In each case every prefix cut has size at most two.
This proves the auxiliary assertion.

Apply it to $H-e$, which has $h-1$ edges. For every prefix $P_j$, the
expression in \textup{(iii)} is at most
$\floor{2(h-1)/3}+2\le2\floor{h/3}+2=2k+2$.
\end{proof}
Next we record two elementary identities. Let $R$ be a subdivision of
$H$, and let $V(R)=A_R\cup B_R$. Denote by $(U_A,U_B)$ the
corresponding partition of the branch vertices of $R$, and set
$$
D:=|U_A|+e_H(U_B)-|U_B|-e_H(U_A),\qquad
t_A:=e(R[A_R]),\qquad t_B:=e(R[B_R]).
$$
Thus, $t_A$ and $t_B$ count the edges of $R$ lying entirely in the two
parts. We have
\begin{equation}\label{eq:balance}
 |A_R|-|B_R|=D+t_A-t_B.
\end{equation}
Indeed,
$$
\begin{aligned}
|A_R|-|B_R|-t_A+t_B
 &=\frac12\left(
   \sum_{v\in A_R}(2-d_R(v))
   -\sum_{v\in B_R}(2-d_R(v))
   \right)\\
 &=\frac12\left(
   \sum_{v\in U_A}(2-d_H(v))
   -\sum_{v\in U_B}(2-d_H(v))
   \right)\\
 &=D.
\end{aligned}
$$
The second equality holds because every internal vertex of $R$ has degree
two, while the final equality follows from
$\sum_{v\in U_A}d_H(v)=2e_H(U_A)+e_H(U_A,U_B)$ and the analogous identity
for $U_B$.

For each $i\in[h]$, let $R_i$ be the path of $R$ replacing $e_i$, set
$s_i=e(R_i)$, and let $t_i$ be the number of edges of $R_i$ lying entirely
in one of the two parts. Also, let $\chi_i=1$ if the branch vertices
corresponding to the ends of $e_i$ lie in different parts, and let
$\chi_i=0$ otherwise. Since $R_i$ has $s_i-t_i$ crossing edges, and this
number is odd precisely when $\chi_i=1$, we have
\begin{equation}\label{eq:path}
 t_i\equiv s_i-\chi_i\pmod2.
\end{equation}

We shall also use the following two Hamilton path results.

\begin{lemma}\label{hpaths}
The following statements hold.
\begin{enumerate}[label=\textup{(\roman*)},leftmargin=*]
\item Let $J$ be an $m$-vertex graph with
$\delta(J)\ge(m+1)/2$. Then, for every two distinct vertices
$x,y\in V(J)$, the graph $J$ contains a Hamilton $x$--$y$ path.

\item Let $J=(X,Y)$ be a bipartite graph with
$|X|=|Y|=m$ and $\delta(J)\ge m/2+1$. Then, for every
$x\in X$ and $y\in Y$, the graph $J$ contains a Hamilton
$x$--$y$ path.
\end{enumerate}
\end{lemma}

\begin{proof}
We begin with \textup{(i)}. We use the following closure observation.
Let $F$ be an $N$-vertex graph, and let $u,v$ be nonadjacent vertices
satisfying $d_F(u)+d_F(v)\ge N$. Then $F$ is Hamiltonian if and only if
$F+uv$ is Hamiltonian.

Only one direction requires proof. Suppose that a Hamilton cycle of
$F+uv$ uses $uv$. Deleting this edge gives a Hamilton path
$u=z_1z_2\ldots z_N=v$. The two sets
$$
\{i\in[N-1]:uz_{i+1}\in E(F)\}
\quad\text{and}\quad
\{i\in[N-1]:vz_i\in E(F)\}
$$
have total size at least $N$, and hence intersect. The corresponding two
edges close the path into a Hamilton cycle of $F$.

Now fix distinct $x,y\in V(J)$ and add a new vertex $w$ adjacent only to
$x$ and $y$. The resulting graph has $m+1$ vertices, and the ends of
every missing edge inside $V(J)$ have degree sum at least $m+1$.
We may therefore add all such edges without changing Hamiltonicity.
The resulting graph consists of $K_m$ together with a vertex $w$ adjacent
to $x$ and $y$, and is Hamiltonian. Reversing the closure operations and
then deleting $w$ gives a Hamilton $x$--$y$ path in $J$.

For \textup{(ii)}, we use the bipartite analogue of the same closure
observation. Let $F=(P,Q)$ be balanced with $|P|=|Q|=M$, and let
$u\in P$ and $v\in Q$ be nonadjacent. If
$d_F(u)+d_F(v)\ge M+1$, then adding $uv$ does not change Hamiltonicity.
Indeed, after deleting $uv$ from a Hamilton cycle, the same switching
argument applies to the $M$ possible reconnection positions.

Fix $x\in X$ and $y\in Y$. Add vertices $x'$ and $y'$, let
$X'=X\cup\{y'\}$ and $Y'=Y\cup\{x'\}$, and add the edges
$xx'$, $x'y'$ and $y'y$. The new bipartition classes have size $m+1$.
For every missing edge between $X$ and $Y$, the sum of the degrees of its
ends is at least $m+2$. We may therefore complete the pair $(X,Y)$ to
$K_{m,m}$ without changing Hamiltonicity.

The completed graph is Hamiltonian: take a Hamilton $x$--$y$ path in
$K_{m,m}$ and add the three edges $xx'$, $x'y'$ and $y'y$. Since
$x'$ and $y'$ have degree two throughout, reversing the closure operations
yields a Hamilton cycle that still contains these three edges. Deleting
$x'$ and $y'$ then gives the required Hamilton $x$--$y$ path in $J$.
\end{proof}

\begin{lemma}\label{linearforest}
Let $J$ be a graph, let $R\subseteq V(J)$, and let $d$ be a non-negative integer. Suppose that $|R|>6d$ and $d_J(x)\ge d$ for every $x\in R$. Then, for every integer $q$ with $0\le q\le2d$, the graph $J$ contains a linear forest with exactly $q$ edges.
\end{lemma}

\begin{proof}
The assertion is trivial when $d=0$. We first record a simple consequence of Hall's theorem. Suppose that $L=(P,Q)$ is a bipartite graph with $|P|\ge 2d$ and
$d_L(x)\ge d$ for every $x\in P$. Then $L$ contains a linear forest
with $2d$ edges.

Indeed, choose $P_0\subseteq P$ with $|P_0|=2d$, and replace every
vertex of $Q$ by two copies. More precisely, let
$Q^\ast=Q\times\{1,2\}$, where $x\in P_0$ is adjacent to $(y,i)\in Q^\ast$
whenever $xy\in E(L)$. For every nonempty $X\subseteq P_0$,
$|N_{Q^\ast}(X)|=2|N_L(X)|\ge2d\ge|X|$. Hence Hall's theorem yields a matching saturating $P_0$.

Projecting the two copies of each vertex of $Q$ back onto $Q$, we
obtain a subgraph $H\subseteq L$ with $2d$ edges such that every
vertex of $P_0$ has degree one and every vertex of $Q$ has degree at
most two. Thus $\Delta(H)\le2$, and $H$ contains no cycle, since every
cycle would contain a vertex of $P_0$ of degree two. Hence $H$ is a
linear forest.

Now let $F\subseteq J$ be a spanning linear forest with the maximum
possible number of edges. Suppose that $e(F)<2d$. The nontrivial
components of $F$ contain at most $2e(F)<4d$ vertices. Since
$|R|>6d$, there are more than $2d$ vertices of $R$ that are isolated
in $F$, let $I$ denote this set.

By the maximality of $F$, every neighbor of a vertex of $I$ is an
internal vertex of a nontrivial component of $F$. Let $U$ be the set
of all such internal vertices. Hence every $x\in I$ has at least $d$
neighbors in $U$. Applying the preceding observation to the bipartite graph $J[I,U]$,
we obtain a linear forest with $2d$ edges, contradicting the choice
of $F$. Therefore $e(F)\ge2d$.

Finally, repeatedly deleting an end edge from a nontrivial component
of $F$ gives a linear forest with exactly $q$ edges for every
$0\le q\le2d$.
\end{proof}

The following allocation lemma follows from hypergeometric concentration.

\begin{lemma}\label{allocation}
Fix $K_0\in\mathbb N$ and $c,\rho>0$. Then, for every sufficiently large
integer $N$, the following holds. Let $Z$ be a set of size $N$, and let $W$ be a set with $|W|\le2N$. For each $v\in W$, let $F(v)\subseteq Z$ satisfy $|F(v)|\le\rho N$.
Suppose that $s\le K_0$ and that $k_1,\ldots,k_s$ are positive integers
satisfying $k_1+\cdots+k_s=N$ and $k_i\ge cN$ for every $i\in[s]$.
Then $Z$ has a partition
$Z=Z_1\cup\cdots\cup Z_s$ with $|Z_i|=k_i$ such that
\begin{equation}\label{eq:prob}
 |F(v)\cap Z_i|\le2\rho k_i+N^{2/3}
\end{equation}
for every $v\in W$ and $i\in[s]$.
\end{lemma}

\begin{proof}
Choose uniformly at random a partition of $Z$ with the prescribed class
sizes. For fixed $v\in W$ and $i\in[s]$, the random variable
$X_{v,i}=|F(v)\cap Z_i|$ is hypergeometric and satisfies
$\mathbb EX_{v,i}\le\rho k_i$. Since the threshold in
\eqref{eq:prob} exceeds its expectation by at least $N^{2/3}$,
Hoeffding's inequality gives
$\mathbb P(X_{v,i}>2\rho k_i+N^{2/3})
\le\exp(-2N^{4/3}/k_i)\le\exp(-2N^{1/3})=o(N^{-2})$.
There are at most $2K_0N$ choices of $(v,i)$, so the assertion follows
from the union bound.
\end{proof}

\subsection{Extremal Case 1: nearly bipartite graphs}
\label{subsec:extremal1}

\begin{proof}[Proof of \cref{thm:main} in Extremal Case~1]
\mbox{}\par

We first construct a small $H$-subdivision in which every short
replacement path has its prescribed length and every long replacement
path contains a marked regular crossing edge. The corrections inside
$A$ and $B$ are chosen so that the subdivision has the same imbalance
as the partition $A\cup B$. The unused vertices in the two parts
therefore have equal size and can be distributed among the long
replacement paths, which are then completed through their marked
crossing edges.

\medskip
\noindent\textbf{Step 1. Cleaning the nearly bipartite graph.}
\medskip

Choose $A_0\subseteq V(G)$ such that
$|A_0|\ge n/2-\varepsilon n$ and $e(A_0)\le\varepsilon n^2$.
We first obtain a balanced bipartition of $V(G)$. If
$|A_0|\ge\floor{n/2}$, choose
$X\subseteq A_0$ with $|X|=\floor{n/2}$. Otherwise, enlarge $A_0$
to a set $X$ of order $\floor{n/2}$. In the latter case, at most
$\varepsilon n$ vertices are added, and hence
$e(X)\le e(A_0)+\varepsilon n^2\le2\varepsilon n^2$.
Thus $e(X)\le2\varepsilon n^2$ in either case. Let
$Y=V(G)\setminus X$, so $|Y|=\ceil{n/2}$.

Since $\delta(G)\ge |Y|+k$, we have $e(X,Y)=\sum_{x\in X}d_G(x)-2e(X) \ge |X|(|Y|+k)-4\varepsilon n^2$, and hence
\begin{equation}\label{eq:case1missing}
 \ole(X,Y)\le4\varepsilon n^2.
\end{equation}

For $c\in\{1,10\}$, define $X_c=\{x\in X:\oldeg_Y(x)>c\eta n\}$ and
$Y_c=\{y\in Y:\oldeg_X(y)>c\eta n\}$. Let $E=X_1\cup Y_1$, $s=|E|$, and
$s_0=|X_{10}|+|Y_{10}|$. Since each missing pair between $X$ and $Y$ is counted once from each side, $c\eta n(|X_c|+|Y_c|)<2\ole(X,Y)$. It follows from
\eqref{eq:case1missing} that
\begin{equation}\label{eq:case1exceptional}
 s\le\frac{8\varepsilon}{\eta}n,
 \qquad
 s_0\le\frac{4\varepsilon}{5\eta}n.
\end{equation}
In particular, $s,s_0\le\eta^2n$.

Move the vertices of $X_{10}$ to $Y$ and those of $Y_{10}$ to $X$.
Denote the resulting parts by $A$ and $B$, interchanging their names
if necessary so that $a:=|A|\le |B|=:b$, and write $r=b-a$. Since
the original part sizes differ by at most one and each moved vertex
changes their difference by two,
\begin{equation}\label{eq:case1r}
 0\le r\le2s_0+1\le3\eta^2n.
\end{equation}

Call the vertices of $E$ \emph{exceptional} and all other vertices regular,
and write $\Reg=V(G)\setminus E$. Every regular vertex remains in its
original part and has at most
$\eta n+s_0\le2\eta n$ non-neighbors in the opposite new part.

We next consider the exceptional vertices. If $x\in X_{10}$, then
$d_Y(x)<|Y|-10\eta n$, and therefore
$d_X(x)\ge\delta(G)-d_Y(x)>10\eta n+k$. After moving $x$ to the
other part, it has at least
$d_X(x)-|X_{10}|\ge8\eta n$ neighbors in the opposite part.
If $x\in X_1\setminus X_{10}$, then $x$ is not moved and retains at
least $|Y|-10\eta n-|Y_{10}|\ge8\eta n$ neighbors in the opposite
part. The same argument applies to vertices originally in $Y$.

Consequently, every exceptional vertex has at least $8\eta n$
neighbors in the opposite part. Since $|E|\le\eta^2n$, every vertex
has at least $7\eta n$ regular neighbors in the opposite part.

The new partition remains almost complete bipartite. Moving $s_0$
vertices changes the crossing status of at most $s_0n$ pairs, so
\eqref{eq:case1missing} and
\eqref{eq:case1exceptional} give
\begin{equation}\label{eq:case1missingedges}
 \ole(A,B)
 \le4\varepsilon n^2+s_0n
 \le\frac{5\varepsilon}{\eta}n^2
 =:\zeta n^2,
 \qquad
 \zeta\ll\eta^2.
\end{equation}

We also record lower bounds on the degrees inside the two parts. For
$v\in B$ and $u\in A$, we have
$d_B(v)\ge\delta(G)-a$ and $d_A(u)\ge\delta(G)-b$. Since
$a+b=n$, $b-a=r$, and $\delta(G)\ge n/2+k$, the integrality of the
degrees gives
\begin{equation}\label{eq:case1degree}
 \delta(G[B])\ge k+\ceil{r/2},
 \qquad
 \delta(G[A])\ge\max\{0,k-\floor{r/2}\}.
\end{equation}

We shall repeatedly use two consequences of these estimates. First,
after excluding at most $\eta n$ previously used vertices, every set of
at most $h$ regular vertices in one part has a fresh regular common
neighbor in the other. Indeed, each such vertex has at most
$2\eta n$ non-neighbors across the partition, while both parts have
order $n/2+o(n)$ and $|E|\le\eta^2n$.

Second, if $U\subseteq A\cap\Reg$ and
$W\subseteq B\cap\Reg$ satisfy $|U|,|W|\ge6\eta n$, then
$G[U,W]$ contains an edge. Otherwise, every pair in $U\times W$
would be missing, and \eqref{eq:case1missingedges} would imply
$\zeta n^2\ge|U||W|\ge36\eta^2n^2$, contradicting
$\zeta\ll\eta^2$.

\medskip
\noindent\textbf{Step 2. Choosing the branch coloring and the correction scheme.}
\medskip

There is at least one long edge. Indeed, otherwise
$\sum_{i=1}^h\ell_i<h\alpha n$, whereas
$\sum_{i=1}^h\ell_i=n-|V(H)|+h$, a contradiction for sufficiently
large $n$. Fix a long edge $e_j$. For each $i\in[h]$, let
$p_i\in\{0,1\}$ satisfy $p_i\equiv\ell_i\pmod2$.

A \emph{correction} for $e_i$ in a side $Z\in\{A,B\}$ is an edge
of the path replacing $e_i$ that lies entirely in $Z$.
Corrections sharing a branch vertex form a \emph{fan port}.
A correction not attached to a branch vertex forms a \emph{single port},
while two corrections assigned to $e_j$ form a \emph{pair port}.
The \emph{demand} of a port is its number of corrections.

For a branch coloring $f:V(H)\to\{A,B\}$, let $\chi_i=1$ if the
ends of $e_i$ receive different colors and let $\chi_i=0$ otherwise. Write $D:=D_f$, let $t_i$ denote the number of corrections assigned to $e_i$, and let $t_A$ and $t_B$ denote the total numbers of corrections assigned to $A$ and $B$, respectively. We shall choose $f$ and the correction scheme so that
\begin{equation}\label{eq:case1correction}
 t_i\equiv p_i-\chi_i\pmod2 \text{for every }i\in[h],\qquad t_B-t_A=r+D.
\end{equation}
The first condition gives the required parity of each replacement path,
while the second gives the required balance between the two parts.

Define
\begin{equation}\label{eq:case1dAB}
 d_A:=\max\{0,k-\floor{r/2}\},
 \qquad
 d_B:=k+\ceil{r/2}.
\end{equation}
We shall use at most $d_A$ ports in $A$ and at most $d_B$ ports in
$B$, each of demand at most $h$.

We shall also use the following parity relation. Let $z=z_p(f)$ be
the number of defects. Since $\sum_i\chi_i=e_H(U_A,U_B)$, we have
$z\equiv\sum_i(p_i-\chi_i) \equiv\sum_i p_i-e_H(U_A,U_B)\pmod2$.
Moreover, $\sum_i p_i\equiv n-v(H)+h\pmod2$, while, by the definition of $D$,
$D\equiv v(H)+h-e_H(U_A,U_B)\pmod2$. Finally, $r\equiv n\pmod2$, since the two parts have orders $(n-r)/2$ and $(n+r)/2$. Hence $z\equiv r+D\pmod2$.

\smallskip
\noindent\textbf{Case 1: $r<2h$.}
Construct a subdivision $K$ of $H$ by leaving $e_i$ unsubdivided when
$p_i=1$ and subdividing it once when $p_i=0$. Thus the path of $K$
corresponding to $e_i$ has length congruent to $p_i$ modulo $2$.

Let $\nu$ be the largest integer such that $K$ contains two disjoint
independent sets of order $\nu$. Then
$\operatorname{bal}(K)=v(K)-2\nu$. Indeed, choose
$T\subseteq V(K)$ and a bipartition $(P,Q)$ of $K-T$ such that
$\bigl||P|-|Q|\bigr|=\disc(K-T)$. Then $|T|+\disc(K-T) =v(K)-2\min\{|P|,|Q|\}$,
and $P$ and $Q$ contain disjoint independent sets of common order
$\min\{|P|,|Q|\}$. Conversely, two disjoint independent sets of common
order $q$ induce a balanced bipartite subgraph of $K$ on $2q$ vertices.

Write $\lambda=\operatorname{bal}(K)$. By \cref{correction}\textup{(i)},
$\lambda\le\floor{2h/3}$. Furthermore, $v(K)=v(H)+h-\sum_i p_i\equiv n\equiv r\pmod2$, and hence $\lambda\equiv r\pmod2$. It follows that
$\lambda\le2k+\1\{r\text{ is odd}\}$.

Choose disjoint independent sets $I_A,I_B\subseteq V(K)$ of order
$\nu$. If one of them can be enlarged while remaining disjoint from
the other, interchange their names if necessary so that $I_B$ can be
enlarged. Extend $I_B$ to a maximal independent set $I'_B$ of
$K-I_A$. Let $J=I'_B\setminus I_B$, $q=|J|$, and
$S=V(K)\setminus(I_A\cup I'_B)$. Since $v(K)=2\nu+\lambda$, we have $|S|=\lambda-q$.

Every vertex of $S$ has a neighbor in $I'_B$ by maximality. It also
has a neighbor in $I_A$. Indeed, if $q>0$ and some $v\in S$ had no
neighbor in $I_A$, then $I_A\cup\{v\}$ and a $(\nu+1)$-subset of
$I'_B$ would be disjoint independent sets of order $\nu+1$, contrary
to the definition of $\nu$. If $q=0$, neither $I_A$ nor $I_B$ is
extendable.

For an integer $x$ to be chosen below, color exactly $x$ vertices of
$S$ with $A$, color the remaining vertices of $S$ with $B$, and
color $I_A$ and $I'_B$ with $A$ and $B$, respectively. Let $f$ be
the restriction of this coloring to $V(H)$. Since $I_A$ and $I'_B$
are independent, every monochromatic edge of $K$ has an endpoint in
$S$. Assign each such edge to one of its endpoints in $S$.

If the selected endpoint is a branch vertex, assign the corresponding
correction to the fan port centered at that vertex. If the selected
endpoint is a subdivision vertex, then its two neighbors lie one in
$I_A$ and one in $I'_B$. Hence exactly one of its incident edges is
monochromatic, and we assign the corresponding correction to the fan
port at the appropriate branch end. The only way that both edges replacing some subdivided $e_i$ are monochromatic is that its subdivision vertex lies in
$I_A\cup I'_B$ and both branch ends have the same color as that
vertex. In this case, $e_i$ receives one fan correction at each branch
end, and the two corrections lie in the same side.

Before adding pair ports, there are at most $x$ ports in $A$ and at most $|S|-x$ ports in $B$. By \eqref{eq:balance}, the difference between the two color
classes of $K$ is $D+t_A-t_B=2x-\lambda$. Let $y_A$ and $y_B$ denote the numbers of pair ports added in $A$ and $B$, respectively. Since $\lambda\equiv r\pmod2$, all the quantities below are integers.

\begin{enumerate}[label=\textup{(\roman*)},leftmargin=*]
\item If $r>\lambda$, choose $x=0$, $y_A=0$, and
$y_B=(r-\lambda)/2$. The numbers of ports in $A$ and $B$ are then at
most $0$ and $(\lambda+r)/2-q$, respectively.

\item If $r\le\lambda$ and $q\le(\lambda+r)/2$, choose
$x=(\lambda-r)/2$ and $y_A=y_B=0$. The numbers of ports are at most
$(\lambda-r)/2$ in $A$ and $(\lambda+r)/2-q$ in $B$.

\item If $r\le\lambda$ and $q>(\lambda+r)/2$, choose
$x=|S|=\lambda-q$, $y_A=q-(\lambda+r)/2$, and $y_B=0$. The numbers
of ports are at most $(\lambda-r)/2$ in $A$ and $0$ in $B$.
\end{enumerate}

In all three cases, $0\le x\le|S|$ and $y_A,y_B\ge0$, and $2x-\lambda+2y_A-2y_B=-r$. Thus the second relation in \eqref{eq:case1correction} holds.
Furthermore, since
$\lambda\le2k+\1\{r\text{ is odd}\}$ and $\lambda\equiv r\pmod2$,
we have $(\lambda+r)/2\le d_B$ and $(\lambda-r)/2\le d_A$ whenever $r\le\lambda$. Hence the numbers of ports do not exceed the bounds in
\eqref{eq:case1dAB}.

For each $e_i$, the number of monochromatic edges on its corresponding
path in $K$ is congruent to $p_i-\chi_i$ modulo $2$. Adding pair ports
changes the number of corrections on $e_j$ by an even number, so the
first relation in \eqref{eq:case1correction} also holds. Every fan
port has demand at most $d_H(v)\le h$, while every pair port has demand
two. Finally, each $e_i\ne e_j$ receives at most one correction, except
possibly one fan correction at each branch end as described above.

\smallskip
\noindent\textbf{Case 2: $r\ge2h$.}
By \cref{correction}\textup{(ii)}, choose a branch
coloring $f$ such that $z+D\le2k$, where $z=z_p(f)$ and $D=D_f$. Moreover,
$|D|\le\frac12\sum_{v\in V(H)}|2-d_H(v)|\le h$. Since $z\le h$, we
have $z-D\le2h\le r$.

We place all corrections in $B$. Set
$t_A=0$ and $t_B=r+D$. Then $t_B\ge z$, and the parity relation above
gives $t_B\equiv z\pmod2$. Furthermore,
\begin{equation}\label{eq:case1bound}
 t_B+z=r+(D+z)
 \le r+2k
 \le2\bigl(k+\ceil{r/2}\bigr)
 =2d_B.
\end{equation}

Let $z_0$ be the number of defect edges different from $e_j$. Assign one
single port to each of these edges and assign the remaining
$t_B-z_0$ corrections to $e_j$. Since $t_B\equiv z\pmod2$, the integer $t_B-z_0$ is odd exactly when $e_j$ is a defect.  Hence these corrections can be grouped into pair
ports, together with one single port when $e_j$ is a defect. The total
number of ports in $B$ is $z_0+\ceil{(t_B-z_0)/2}=(t_B+z)/2\le d_B$, and no port is used in $A$.

Combining the two cases, the correction scheme satisfies \eqref{eq:case1correction} and uses at most $d_A$ ports in $A$ and at most $d_B$ ports in $B$. Every port has demand at most $h$, and the total demand is $O_h(r+1)$.

\medskip
\noindent\textbf{Step 3. Realizing the corrections inside the two parts.}
\medskip

For $Z\in\{A,B\}$, let $\overline Z$ denote the other part, and let
$d_Z$ be defined by \eqref{eq:case1dAB}. Define $Q_Z$ to be the graph
obtained from $G[Z]$ by deleting all edges with two exceptional
endpoints. Then \eqref{eq:case1degree} gives
$d_{Q_Z}(v)\ge d_Z$ for every $v\in Z\cap\Reg$.

\begin{claim}\label{matchingfan}
Suppose that $r<2h$, and let $Z\in\{A,B\}$ with $d_Z>0$. Then one of
the following holds.
\begin{enumerate}[label=\textup{(\alph*)},leftmargin=*]
\item $Q_Z$ contains a matching of size $(2h+4)d_Z$.
\item There are disjoint sets $C_Z,L_Z\subseteq Z$ such that
$|C_Z|=d_Z$, $L_Z\subseteq\Reg$, $|L_Z|=\Omega_h(n)$, and
$C_Z\subseteq N_{Q_Z}(v)$ for every $v\in L_Z$.
\end{enumerate}
\end{claim}

\begin{proof}[Proof of the claim]
Suppose that \textup{(a)} fails, and let $M_Z$ be a maximal matching
in $Q_Z$. Its set of endpoints $W_Z$ is a vertex cover and satisfies
$|W_Z|<2(2h+4)d_Z=O_h(1)$. Every vertex of
$(Z\cap\Reg)\setminus W_Z$ has at least $d_Z$ neighbors in $W_Z$.
Since $(Z\cap\Reg)\setminus W_Z$ has linear order and $W_Z$ has only
$O_h(1)$ subsets of order $d_Z$, some $d_Z$-set
$C_Z\subseteq W_Z$ is contained in the neighborhoods of
$\Omega_h(n)$ vertices of $(Z\cap\Reg)\setminus W_Z$. Taking these
vertices as $L_Z$ proves \textup{(b)}.
\end{proof}

If \textup{(a)} holds, call $Z$ a \emph{matching side} and fix a
matching of size $(2h+4)d_Z$. Otherwise, call $Z$ a \emph{fan side}
and fix sets $C_Z$ and $L_Z$ as in \textup{(b)}.

\begin{claim}
Suppose that $r<2h$. The correction scheme may be modified, while
preserving \eqref{eq:case1correction}, so that every
$e_i\ne e_j$ has one of the following forms:
\begin{enumerate}[label=\textup{(\alph*)},leftmargin=*]
\item no correction;
\item one fan correction at one branch end;
\item one detached correction in a matching side, with both branch
images free;
\item one fan correction at each branch end, with both corrections
lying in the same side.
\end{enumerate}
Moreover, no new fan port is created, and every fan port has demand at
most $h$.
\end{claim}

\begin{proof}[Proof of the claim]
Detach every correction lying in a matching side from its branch
center, and consider $e_i\ne e_j$.

If both endpoint corrections are detached, transfer both to $e_j$ in
the same side. If exactly one correction is detached and the other
branch end is not fixed as a fan center, leave the detached correction
on $e_i$.

It remains to consider the case in which exactly one correction is
detached, say in the matching side $Z$, while the other branch end is
already the center of a fan in $\overline Z$. Replace the detached
correction on $e_i$ by a correction in this fan, and add two
corrections on $e_j$ in $Z$.

The first modification leaves the side totals unchanged, removes two
corrections from $e_i$, and adds two corrections to $e_j$. In the
second modification, the numbers of corrections in $Z$ and
$\overline Z$ both increase by one, the number assigned to $e_i$ is
unchanged, whereas the number assigned to $e_j$ increases by two.
Hence \eqref{eq:case1correction} is preserved.

In the second modification, the arm $e_i$ was not previously
represented in the fan to which its correction is added. Thus no new
fan port is created, and a fan centered at a branch vertex $v$ has
demand at most $d_H(v)\le h$. The four stated forms follow.
\end{proof}

\smallskip
\noindent\textbf{The case $r<2h$.}
We now realize the normalized corrections. Let $Z$ be a matching side.
Before normalization, at most $d_Z$ ports are assigned to $Z$, each of
demand at most $h$, so there are at most $hd_Z$ corrections in $Z$.
Each modification above increases the number of corrections in a
matching side by at most one, and at most $h$ such modifications are
performed. Since $d_Z\ge1$, the number of corrections in $Z$ is at
most $hd_Z+h\le(2h+4)d_Z$.

Assign a distinct edge of the fixed matching in $Q_Z$ to each
correction. In particular, the two corrections of a pair port are
represented by two distinct matching edges. Every chosen edge has a
regular endpoint, since $Q_Z$ contains no edge with two exceptional
endpoints.

Now let $Z$ be a fan side. The normalization creates no new fan port,
so at most $d_Z$ ports are assigned to $Z$. Assign a distinct vertex
of $C_Z$ to each port. For a fan port, use the assigned vertex as the
image of its branch center and choose a distinct leaf from $L_Z$ for
each correction in the fan. For a pair port, choose two distinct
leaves from $L_Z$ and join them by the two-edge path through its
assigned vertex of $C_Z$. Since $|L_Z|=\Omega_h(n)$ and the total
demand is $O_h(1)$, all leaves can be chosen distinct.

Thus the resulting correction pieces are pairwise vertex-disjoint,
except that pieces belonging to the same fan port share their
prescribed branch image.

\smallskip
\noindent\textbf{The case $r\ge2h$.}
Let $d=d_B$. Every regular vertex has degree at least $d$ in $Q_B$.
Moreover, Step~1 and the hierarchy give
$|B\setminus E|>6d$.

Let $z_0$ be the number of defects other than $e_j$. By
\eqref{eq:case1bound},
$t_B+z_0\le t_B+z\le2d$. Applying
\cref{linearforest} to $Q_B$, with
$R=B\setminus E$ and $q=t_B+z_0$, yields a linear forest
$F_B\subseteq Q_B$ with exactly $t_B+z_0$ edges.

Since $t_B\ge z\ge z_0$, we have $e(F_B)\ge2z_0$. Traverse the path
components of $F_B$ from an end, selecting the first edge and then
every second edge until $z_0$ edges have been selected. Delete the
edge immediately following each selected edge whenever it exists.
The selected edges are pairwise vertex-disjoint and are not incident
with any edge that remains.

At most $2z_0$ edges have been selected or deleted, so at least
$t_B-z_0$ edges remain. Delete end edges from the remaining path
components until exactly $t_B-z_0$ edges remain. Assign one selected
edge to each defect other than $e_j$, and assign every component of
the remaining linear forest to $e_j$. These correction pieces are
pairwise vertex-disjoint, and every selected edge has a regular
endpoint.

Finally, choose the image of every branch vertex that is neither
already fixed as a fan center nor incident with a detached correction
on some $e_i\ne e_j$ as a fresh regular vertex in its prescribed
part. The images of the remaining branch vertices will be chosen in
Step~4 after the relevant adjacency constraints have been specified.
Thus all correction pieces are pairwise vertex-disjoint, except at
their prescribed common branch images.

\medskip
\noindent\textbf{Step 4. Embedding the paths corresponding to $e_i\ne e_j$.}
\medskip

For $X\in\{A,B\}$, write $\overline X$ for the other part. For each
$i\in[h]\setminus\{j\}$, we construct a path $R_i$ of length four or
five corresponding to $e_i$, choosing any branch images not fixed in
Step~3 during the construction. The edges of $R_i$ lying within $A$
or $B$ are exactly the corrections assigned to $e_i$.

\smallskip
\noindent\textbf{Detached corrections.}

Let $xy\in E(G[X])$ be a detached correction, where $X\in\{A,B\}$, and label its ends so that $y$ is regular. By the normalization in Step~3, the images of both ends of $e_i$ have not yet been chosen.

Suppose first that the ends of $e_i$ are colored $X$ and
$\overline X$, respectively. Choose a fresh regular vertex
$c\in N_{\overline X}(x)$. Require the branch image $u\in X$ to be
adjacent to $c$ and the branch image $v\in\overline X$ to be adjacent
to $y$. Once $u$ and $v$ have been chosen, use the path $(u,c,x,y,v)$. The reverse coloring is handled symmetrically.

If both ends of $e_i$ are colored $X$, choose distinct fresh regular
vertices $c\in N_{\overline X}(x)$ and $d\in N_{\overline X}(y)$.
Require the two branch images $u,v\in X$ to be adjacent to $c$ and
$d$, respectively, and use the path $(u,c,x,y,d,v)$.

If both ends of $e_i$ are colored $\overline X$, choose a fresh regular vertex $c\in N_{\overline X}(x)$ and then a fresh regular vertex $z\in N_X(c)$. Require the two branch images $u,v\in\overline X$ to be adjacent to $y$ and $z$, respectively, and use the path $(u,y,x,c,z,v)$.

Process all detached corrections before choosing the remaining branch
images. Each resulting constraint is a regular vertex in the part
opposite to the prescribed branch image. A branch vertex receives at
most one constraint from each incident edge, and hence at most
$d_H(v)\le h$ constraints. Since only $O_h(r+1)\ll\eta n$ vertices are used, 
the common neighbor property from Step~1 allows the remaining branch vertices to be embedded successively as distinct fresh regular common neighbors satisfying all their
constraints.

\smallskip
\noindent\textbf{Retained fan corrections.}

Suppose first that $e_i$ has one retained fan correction
$ux\in E(G[X])$, where $u$ is the branch image at the center of the
fan and $x$ is its regular leaf. Let $v$ be the other branch image.

If $v\in\overline X$, let $T$ be the set of vertices already used
outside the current correction and its branch images. Then
$U=(N_{\overline X}(x)\cap\Reg)\setminus T$ and  $W=(N_X(v)\cap\Reg)\setminus T$
have order at least $6\eta n$. By the second consequence of Step~1,
there is an edge $yz$ with $y\in U$ and $z\in W$. Use the path
$(u,x,y,z,v)$.

If $v\in X$, choose distinct fresh regular vertices $y\in N_{\overline X}(x)$ and
$w\in N_{\overline X}(v)$, and then a fresh regular common neighbor
$z\in X$ of $y$ and $w$. Use the path $(u,x,y,z,w,v)$. The case in which the fan correction lies at the other branch end is symmetric.

Finally, suppose that $e_i$ has a retained fan correction at each
branch end. By Step~3, both corrections lie in the same part, say
$X$. Write them as $ux$ and $vy$, where $u,v$ are the branch images
and $x,y$ are the regular leaves. Choose a fresh regular common
neighbor $z\in\overline X$ of $x$ and $y$, and use the path
$(u,x,z,y,v)$.

In each case, the edges of the resulting path lying inside $A$ or $B$
are precisely the retained fan corrections assigned to $e_i$, and the
path contains a crossing edge whose endpoints are regular.

\smallskip
\noindent \textbf{Edges with no correction.}

Suppose that $e_i$ has no correction, and let $u$ and $v$ be its
branch images. If $u,v\in X$, choose distinct fresh regular vertices
$y\in N_{\overline X}(u)$ and $w\in N_{\overline X}(v)$, and then a fresh regular common neighbor $z\in X$ of $y$ and $w$. Use the path $(u,y,z,w,v)$. If $u\in A$ and $v\in B$, choose fresh regular vertices $y\in N_B(u)$ and $z\in N_A(v)$. Next choose a fresh regular vertex $a\in N_A(y)$ and a fresh regular common neighbor $b\in N_B(a)\cap N_B(z)$. Use the path $(u,y,a,b,z,v)$. The case $u\in B$ and $v\in A$ is symmetric.

In Step~4, only $O_h(1)$ auxiliary vertices are added to the correction pieces fixed in Step~3. Since the latter use $O_h(r+s+1)\ll\eta n$ vertices, all choices above can be made fresh. Thus the paths $R_i$, $i\ne j$, are internally vertex-disjoint and
otherwise meet only at common branch images. Each $R_i$ contains a crossing edge whose two endpoints are regular.

By construction, exactly $t_i$ edges of $R_i$ lie within $A$ or $B$. Hence, by \eqref{eq:path} and \eqref{eq:case1correction}, $|E(R_i)|\equiv\chi_i+t_i\equiv p_i\equiv\ell_i\pmod2$. Since $|E(R_i)|\in\{4,5\}$, the path $R_i$ has length four when
$\ell_i$ is even and length five when $\ell_i$ is odd.

\medskip
\noindent\textbf{Step 5. Constructing $R_j$ and completing the short paths.}
\medskip

The paths $R_i$ for $i\ne j$ constructed in Step~4 are now fixed.
We first construct $R_j$.

At each end of $e_j$, if the corresponding correction is retained in
a fan, take the edge joining the branch image to the regular leaf
assigned to $e_j$ as an endpoint piece, with the branch image as its first vertex. Otherwise, take the branch image as a singleton endpoint
piece. Include every remaining correction path assigned to $e_j$ and
every exceptional vertex not yet used as a singleton piece.
Fix an ordering of the vertices along each nontrivial internal piece,
and order the internal pieces between the two endpoint pieces.

By Step~2, the total demand is $O_h(r+1)$, and the realization in
Step~3 therefore uses $O_h(r+1)$ vertices. Hence all the pieces above
together contain at most $C_h(r+s+1)$ vertices, for some constant
$C_h$. By \eqref{eq:case1exceptional}, \eqref{eq:case1r}, and the
hierarchy, we may assume that $C_h(r+s+1)\le\eta n/100$.

\smallskip
\noindent\textbf{Connecting the pieces.}
By Step~1, every vertex, whether regular or exceptional, has at least
$7\eta n$ regular neighbors in the opposite part.
We use the following connection property. Let $x$ and $y$ be exposed
ends of two vertex-disjoint pieces, and let $F$ be a set of at most
$\eta n/5$ forbidden vertices disjoint from $\{x,y\}$. Then $x$ and
$y$ can be joined by a path of length three if they lie in different
parts, and by a path of length four if they lie in the same part, such
that all internal vertices are fresh and regular and every edge of the
path crosses between $A$ and $B$.

Suppose first that $x\in A$ and $y\in B$. The sets 
$U=(N_B(x)\cap\Reg)\setminus(F\cup\{y\})$  and $W=(N_A(y)\cap\Reg)\setminus(F\cup\{x\})$ have order at least $6\eta n$. By \eqref{eq:case1missingedges}, there is an edge $ba$ with $b\in U$ and $a\in W$. Thus $(x,b,a,y)$ is the required path. The case $x\in B$ and $y\in A$ is symmetric.

If $x,y\in A$, choose distinct regular vertices $b\in N_B(x)\setminus F$ and $c\in N_B(y)\setminus F$. By the common neighbor property from Step~1, $b$ and $c$ have a fresh regular common neighbor $a\in A\setminus(F\cup\{x,y\})$. Thus $(x,b,a,c,y)$ is the required path. The case $x,y\in B$ is symmetric.

Apply this connection property successively to consecutive pieces. At
each step, let $F$ contain all vertices already used by the paths
$R_i$ with $i\ne j$, all vertices of the pieces other than the two
exposed ends currently being joined, and all internal vertices of the
previous connectors. The number of pieces is $O_h(r+s+1)$, and each
connector uses at most three new vertices. Thus, by increasing
$C_h$ if necessary and using $C_h(r+s+1)\le\eta n/100$, we have 
$|F|<\eta n/5$ throughout.

The resulting path $R_j$ joins the prescribed branch images, contains
every correction assigned to $e_j$ and every exceptional vertex not
used earlier, and has no other edge lying inside $A$ or $B$. Moreover,
at least one connector is used, and each connector contains a crossing
edge with two regular endpoints. Fix one such edge as the marked edge
of $R_j$.

Exactly $t_j$ edges of $R_j$ lie within $A$ or $B$. Hence, by
\eqref{eq:path} and \eqref{eq:case1correction},
$|E(R_j)|\equiv\chi_j+t_j\equiv p_j\equiv\ell_j\pmod2$.
The number of pieces and connectors also gives
$|E(R_j)|\le\eta n/10<\alpha n/3\le\ell_j$.

For each $i\ne j$, fix a crossing edge of $R_i$ whose endpoints are
regular, as provided by Step~4. We next extend the paths corresponding
to the short edges.

\smallskip
\noindent \textbf{Extending the short paths.}

Let $xy$ be the marked edge of a path corresponding to a short edge,
where $x\in A\cap\Reg$ and $y\in B\cap\Reg$. Let $T$ be the set of
vertices currently used by all subdivision paths. Since the total
length of the short subdivision paths is at most $\beta n$,
$|E(R_j)|\le\eta n/10$, and the paths corresponding to the other long
edges have bounded length, throughout the extension process we have
$|T|\le\beta n+\eta n/10+O_h(1)<\eta n/4$.

Since $x$ is regular, $|(N_B(x)\cap\Reg)\setminus T| \ge |B|-2\eta n-|E|-|T|>n/4$.
Choose a fresh regular vertex $y'\in(N_B(x)\cap\Reg)\setminus T$. Since $y'$ and $y$ are regular, $|(N_A(y')\cap N_A(y)\cap\Reg)\setminus T|
\ge |A|-4\eta n-|E|-|T|>n/4$. Choose a fresh regular vertex $x'$ from this set. Replace the marked edge $xy$ by the path $(x,y',x',y)$, and take $y'x'$ as the new
marked edge.

This operation uses two new regular vertices, increases the length of
the path by two, and creates no edge lying within $A$ or $B$. For each
short edge $e_i$, the current path $R_i$ has the same parity as
$\ell_i$ and has length at most $\ell_i$. Repeating this operation
therefore extends $R_i$ to length exactly $\ell_i$.

For every long edge, leave the current path unchanged. Writing
$s_i=|E(R_i)|$, we obtain subdivision paths that are pairwise internally
vertex-disjoint and otherwise meet only at common branch images, and
\begin{equation}\label{eq:case1prepaths}
\begin{aligned}
 s_i&=\ell_i
   &&\text{if $e_i$ is short},\\
 s_i&\equiv\ell_i\pmod2
   &&\text{for every }i\in[h],\\
 s_i&\le\alpha n/3
   &&\text{if $e_i$ is long}.
\end{aligned}
\end{equation}
Every exceptional vertex lies on one of these paths, and every path
corresponding to a long edge contains a marked crossing edge whose
endpoints are regular.

\medskip
\noindent\textbf{Step 6. Completing the long paths.}
\medskip

Let $R=\bigcup_{i=1}^h R_i$. The normalization in Step~3 preserves
$t_B-t_A$, so \eqref{eq:case1correction} still gives
$t_B-t_A=r+D$. Hence, by \eqref{eq:balance},
$|A\cap V(R)|-|B\cap V(R)|=D+t_A-t_B=-r=a-b$.
Thus the unused sets $A^*=A\setminus V(R)$ and
$B^*=B\setminus V(R)$ have the same order, write
$N=|A^*|=|B^*|$.

Let $I$ be the set of indices corresponding to long edges. For each
$i\in I$, set $k_i=(\ell_i-s_i)/2$. By
\eqref{eq:case1prepaths}, $k_i$ is an integer and
$k_i\ge\alpha n/3$. Since
$|V(R)|=|V(H)|-h+\sum_{i=1}^h s_i$ and
$\sum_{i=1}^h\ell_i=n-v(H)+h$, while $s_i=\ell_i$ for every short
edge, we have $2N=n-|V(R)|=\sum_{i=1}^h(\ell_i-s_i)=2\sum_{i\in I}k_i$.
Consequently, $N=\sum_{i\in I}k_i$.

For each $i\in I$, label the endpoints of the marked edge of $R_i$
as $x_i$ and $y_i$, where $x_i\in A$ and $y_i\in B$. All vertices
of $A^*\cup B^*$, as well as $x_i$ and $y_i$, are regular. Hence
each relevant vertex has at most $2\eta n$ non-neighbors in the
opposite unused set.

Since $e_j$ is long, $N\ge k_j\ge\alpha n/3$. Also $N\le n$, and
hence $k_i\ge(\alpha/3)N$ for every $i\in I$, while
$2\eta n\le(6\eta/\alpha)N$. Since $|I|\le h$ and $N$ is sufficiently
large, \cref{allocation} applies, together with the hierarchy, to
give partitions $B^*=\cup_{i\in I}B_i$ and $A^*=\cup_{i\in I}A_i$,
where $|A_i|=|B_i|=k_i$, such that every relevant vertex has fewer
than $k_i/10$ non-neighbors in the corresponding part.

For each $i\in I$, consider the bipartite subgraph of $G$ with
classes $A_i\cup\{x_i\}$ and $B_i\cup\{y_i\}$. Each class has order
$k_i+1$, and every vertex has degree greater than $9k_i/10$.
Since $9k_i/10>(k_i+1)/2+1$ for sufficiently large $n$,
\cref{hpaths}\textup{(ii)} gives a Hamilton
$x_i$--$y_i$ path.

Replace the marked edge $x_iy_i$ of $R_i$ by this Hamilton path.
This increases its length by exactly $2k_i$ and uses precisely the
vertices of $A_i\cup B_i$. Doing this for every $i\in I$ gives
$|E(R_i)|=s_i+2k_i=\ell_i$ for every long edge $e_i$.
Together with Step~5, the resulting subdivision paths have the
prescribed lengths, are pairwise internally vertex-disjoint and
otherwise meet only at common branch images, and together cover every
vertex of $G$. Hence they form the required spanning
$H$-subdivision. This completes Extremal Case~1.
\end{proof}

\subsection{Extremal Case 2: two dense parts}
\label{subsec:extremal2}

\begin{proof}[Proof of \cref{thm:main} in Extremal Case~2]

\mbox{}\par
\medskip
\noindent\textbf{Step 1. Cleaning the two dense parts and finding the transition matching.}
\medskip

Choose $A_0\subseteq V(G)$ as in Extremal Case~2, and let
$B_0=V(G)\setminus A_0$. Define $D_A=\{a\in A_0:d_{B_0}(a)>\eta n\}$ and
$D_B=\{b\in B_0:d_{A_0}(b)>\eta n\}$. Since $e(A_0,B_0)\le\varepsilon n^2$,
\begin{equation}\label{eq:case2-bad}
|D_A|+|D_B|\le\frac{2\varepsilon}{\eta}n=:\rho n,
\qquad \rho\ll\eta.
\end{equation}

Move the vertices of $D_A\cup D_B$ to the opposite side, and set
$X=(A_0\setminus D_A)\cup D_B$ and $Y=(B_0\setminus D_B)\cup D_A$.
Let $C_X=A_0\setminus D_A$ and $C_Y=B_0\setminus D_B$ be the two
cores, and call the moved vertices exceptional. If $x\in C_X$, then
$d_{C_X}(x)\ge\delta(G)-\eta n-|D_A|$, and the analogous bound holds
in $C_Y$. Since $|A_0|,|B_0|=n/2\pm\varepsilon n$, the hierarchy gives
$|C_Z\setminus N(z)|\le3\eta n$ for every $z\in C_Z$ and $Z\in\{X,Y\}$.

If $z$ is exceptional in $Z$, then its definition before the move
gives
\begin{equation}\label{eq:case2exceptional}
d_{C_Z}(z)\ge
\eta n-|D_A|-|D_B|
\ge\eta n/2.
\end{equation}
In particular, $|X|,|Y|=n/2+o(n)$.

Let $Q$ be the bipartite graph obtained from $G[X,Y]$ by deleting all
edges whose two endpoints are exceptional. We claim that $Q$ contains
a matching of size $2k+2$. Set
$p_X=\delta(G)-|X|+1$ and $p_Y=\delta(G)-|Y|+1$.
Every vertex of $C_X$ has at least $p_X$ neighbors in $Y$, and every
vertex of $C_Y$ has at least $p_Y$ neighbors in $X$. Moreover,
$p_X+p_Y=2\delta(G)-n+2\ge2k+2$.

Suppose otherwise. By K\H{o}nig's theorem, $Q$ has a vertex cover
$W=W_X\cup W_Y$, where $W_X\subseteq X$, $W_Y\subseteq Y$, and
$|W|<2k+2$. Since both cores have linear order, we may choose
$x\in C_X\setminus W_X$ and $y\in C_Y\setminus W_Y$.
As $W$ is a vertex cover, all neighbors of $x$ in $Q$ lie in $W_Y$,
and all neighbors of $y$ lie in $W_X$. Hence
$|W_Y|\ge\max{0,p_X}$ and $|W_X|\ge\max{0,p_Y}$, so
$|W|\ge\max\{0,p_X\}+\max\{0,p_Y\}\ge p_X+p_Y\ge2k+2$,
a contradiction. Fix a matching $M_Q$ of size $2k+2$ in $Q$.

We next reserve two private core neighbors for each exceptional
vertex. For every $z\in D_B$, choose distinct
$a_z,b_z\in N_{C_X}(z)$, and for every $z\in D_A$, choose distinct
$a_z,b_z\in N_{C_Y}(z)$, so that all chosen vertices are distinct and
avoid $V(M_Q)$. This can be done greedily: by
\eqref{eq:case2exceptional}, every exceptional vertex has at least
$\eta n/2$ neighbors in the corresponding core, whereas
$2(|D_A|+|D_B|)+|V(M_Q)|\le2\rho n+4k+4\ll\eta n$.

Let $R_X\subseteq C_X$ and $R_Y\subseteq C_Y$ be the sets of reserved
vertices. Until Step~5, every further core vertex is chosen outside
$R_X\cup R_Y$. Since $|R_X|+|R_Y|=2(|D_A|+|D_B|)\le2\rho n\ll\eta n$, 
all degree and common neighborhood estimates used below remain valid.

\medskip
\noindent\textbf{Step 2. coloring the subdivision and embedding the transitions.}
\medskip

Let $F$ be the subdivision of $H$ on $n$ vertices obtained by replacing each edge $e_i$ with a path of length $\ell_i$. We shall color
$V(F)$ red and blue so that exactly $|X|$ vertices are red, red
vertices will be embedded in $X$ and blue vertices in $Y$.

There is a long edge $e^*=e_{i^*}$. Write $e^*=u^*v^*$ and let $J=H-e^*$. By
\cref{correction}\textup{(iii)}, there is an ordering
$\pi=(v_1,\ldots,v_{v(H)})$ such that, for every prefix $P$ of $\pi$,
\begin{equation}\label{eq:case2bound}
 e_J(P,V(H)\setminus P)
 +2-\1\{|\{u^*,v^*\}\cap P|=1\}\le2k+2.
\end{equation}

Leave the internal vertices of the path replacing $e^*$ uncolored
for now and use them later to adjust the sizes of the two color
classes. An edge whose endpoints have different colors is called a
\emph{transition edge}. For a path of length at least five with a
transition, we keep at least two internal vertices on each side of
the transition. If the path has length four, one side contains one
internal vertex and the other contains two; either of the two possible
positions may be used.

We construct the coloring as follows. Initially, color every vertex
outside the interior of the path replacing $e^*$ blue. Recolor the
branch vertices one at a time in the order $\pi$. Suppose that $v_j$
is the next branch vertex to be recolored red. For every short path
incident with $v_j$, if its other branch end is still blue, recolor
its internal vertices so that the path has exactly one transition
edge satisfying the condition above. If its other branch end is
already red, recolor the whole path red.

For every long path of $J$ whose other branch end is still blue,
recolor the first two internal vertices at $v_j$ red. If its other
branch end is already red, recolor its two remaining blue internal
vertices red. After these simultaneous recolorings, process the
newly separated long paths one at a time. On the current path, move
its transition toward the blue branch end by recoloring one internal
vertex at a time, stopping when exactly two blue internal vertices
remain at that end.

Record the initial coloring, the coloring after each simultaneous
recoloring, and the coloring after each one-vertex move. By
construction, in every recorded coloring the red branch vertices
form a prefix $P$ of $\pi$. A path corresponding to an edge of $J$
is monochromatic if its branch ends have the same color, and has
exactly one transition edge if its branch ends have different colors.
Moreover, at most one long path has a moving transition, while every
other transition on a long path lies within two edges of a branch end.

Moving a transition by one vertex increases the number of red vertices
by one. A simultaneous recoloring when a branch vertex becomes red
changes this number by at most
$1+2h+\sum_{\ell_i<\alpha n}\ell_i\le\beta n+3h$.
Thus the recorded red counts outside the interior of the path replacing
$e^*$ increase from $0$ to $n-M$, where $M=\ell_{i^*}-1$.

Consider a recorded coloring whose red branch vertices form a prefix
$P$ of $\pi$. The transition edges on the paths corresponding to $J$
are precisely those corresponding to
$E_J(P,V(H)\setminus P)$. The path replacing $e^*$ requires one
transition if its branch ends have different colors and two if they
have the same color. Hence \eqref{eq:case2bound} shows that
the total number of transitions is at most $2k+2$.

Transitions on different subdivision paths are vertex-disjoint. The
only possible pair of transitions on the same path occurs on the path
replacing $e^*$: when two transitions are used there, keep at least
five vertices in the monochromatic interval between them. Thus the
transition edges form a matching.

For every integer $t$ with $5\le t\le M-5$, the interior of the path
replacing $e^*$ can be colored with exactly $t$ red vertices and the
required number of transitions. If its branch ends have different
colors, use one transition. If both ends are blue, use a red middle
interval; if both ends are red, use a blue middle interval. For the
moment, restrict to
$5+h\le t\le M-5-h$.

Let $0=R_0<R_1<\cdots<R_q=n-M$ be the distinct recorded red counts
outside this path interior. Since $M\ge\alpha n-1$, the preceding
bound gives
$R_{a+1}-R_a<M-10-2h$
for sufficiently large $n$. Hence the integer intervals $[R_a+5+h,\;R_a+M-5-h]$
overlap consecutively, and their union contains every integer from
$5+h$ to $n-5-h$. Since $|X|=n/2+o(n)$, we may choose a recorded
coloring and a value of $t$ in the above range so that the resulting
coloring of $F$ has exactly $|X|$ red vertices.

The transition edges form a matching of size at most $2k+2$. Map them
injectively to the matching $M_Q$ fixed in Step~1, with red endpoints
mapped to $X$ and blue endpoints to $Y$.

For each transition on a path of length four, choose between its two
possible positions so that the side containing only one internal
vertex is mapped to a core endpoint of the corresponding edge of
$M_Q$. This is possible because no edge of $Q$ has two exceptional
endpoints. Changing the position on one such path changes the number
of red vertices outside the path replacing $e^*$ by one. If the net
change is $\Delta$, then $|\Delta|\le h$. Replacing $t$ by
$t-\Delta$ restores the total number of red vertices. Since
$5+h\le t\le M-5-h$, the adjusted value still lies in $[5,M-5]$.

Set $\tau_0=\alpha^2$. On each subdivision path, consider its maximal
monochromatic intervals, and let $\mathcal I_{\mathrm{s}}$ be the
family of those having order less than $\tau_0n$. We shall use the
following bound:
\begin{equation}\label{eq:case2intervals}
 \sum_{I\in\mathcal I_{\mathrm{s}}}|I|
 \le2\beta n+3\tau_0n+30h.
\end{equation}
Indeed, the short subdivision paths contribute at most
$\beta n+O_h(1)$ vertices. On every long path of $J$ except possibly
the one with a moving transition, the short monochromatic interval
has bounded order, so these paths contribute only $O_h(1)$ vertices
in total. The possible active long path contributes fewer than
$\tau_0n$ vertices to $\mathcal I_{\mathrm{s}}$, since
$2\tau_0<\alpha$. Finally, the path replacing $e^*$ has at most three
monochromatic intervals, one of which has order greater than
$\tau_0n$ for sufficiently large $n$, so its short intervals contain
fewer than $2\tau_0n$ vertices in total. This proves
\eqref{eq:case2intervals}.

\medskip
\noindent\textbf{Step 3. Embedding the collars and branch vertices.}
\medskip

For each branch end of a subdivision path, call the first internal
vertex its \emph{branch collar}. On each monochromatic side of a
transition whose interval has order at least three, call the neighbor
of the transition endpoint within that interval its \emph{transition
collar}. The two collars may coincide.

The matching $M_Q$ fixes the images of all transition endpoints. If a
monochromatic interval has order two, then its transition endpoint is
also the branch collar. By Step~2, this can occur only on a path of
length four, and that endpoint is mapped to a core vertex. Otherwise,
map the transition collar to a fresh core neighbor of the image of
the transition endpoint. This is possible by \eqref{eq:case2exceptional} 
when the latter is exceptional, and by the fact that every core vertex has at most $3\eta n$ non-neighbors in its core otherwise. Since there are at most $2k+2$ transitions, only $O_h(1)$ such choices are required, and they can be made pairwise distinct and outside $R_X\cup R_Y$.

Now fix a branch vertex $v$. For each path incident with $v$, retain
the image of its branch collar if it has already been embedded as a
transition endpoint or transition collar, otherwise map it to a fresh
vertex in the appropriate core, adjacent to the next embedded
transition collar whenever necessary. These choices are possible since every vertex in the relevant core has at most $3\eta n$ non-neighbors there. The resulting set has order at most $d_H(v)\le h$, and every vertex in it has at most $3\eta n$ non-neighbors in the core corresponding to the color of $v$.
Hence, by the hierarchy, this set has a fresh common neighbor in that
core outside $R_X\cup R_Y$, use it as the image of $v$. Repeating this
greedily embeds all branch vertices and collars.

\medskip
\noindent\textbf{Step 4. Decomposing the monochromatic intervals.}
\medskip

On each subdivision path, consider its maximal monochromatic intervals.
Every interval of order at most four is already completely embedded.
Indeed, an interval with two branch endpoints is a whole subdivision
path and hence has order at least five, while an interval with two
transition endpoints has order at least five by Step~2. Thus any
interval of order at most four has one branch endpoint $b$ and one
transition endpoint $t$, and is of the form
$(b,t)$, $(b,c,t)$, or $(b,c_1,c_2,t)$. These are precisely the
configurations embedded in Step~3, with the possible coincidences of
the branch and transition collars.

For every monochromatic interval $I$ of order at least five, delete
its two endpoints and denote the remaining subpath by $P_I$. Then
$|P_I|=|I|-2\ge3$, and its two endpoints are distinct core collars
already embedded in Step~3. The pairs of endpoints of these subpaths are pairwise disjoint. Let $\mathcal P_X$ and $\mathcal P_Y$ be the
families of red and blue subpaths, respectively.

Let $S_X$ be the set of vertices already embedded in $X$, excluding
the endpoints of the subpaths in $\mathcal P_X$, and write
$Z_X=X\setminus S_X$; define $S_Y$ and $Z_Y$ analogously. Since these
subpaths, together with the remaining embedded vertices, partition
the two color classes,
\begin{equation}\label{eq:case2-task-sums}
 \sum_{P\in\mathcal P_X}|P|=|Z_X|,
 \qquad
 \sum_{P\in\mathcal P_Y}|P|=|Z_Y|.
\end{equation}
Since each subdivision path contributes one initial monochromatic
interval and each transition creates one additional interval, each
of $\mathcal P_X$ and $\mathcal P_Y$ contains at most
$K:=h+2k+2$ subpaths.

Let $\tau_b=\tau_0/2$ and $\sigma=3\beta+4\alpha^2$, and let
$\mathcal P_{\mathrm{s}}$ consist of the subpaths in
$\mathcal P_X\cup\mathcal P_Y$ of order less than $\tau_b n$.
For sufficiently large $n$, every such subpath comes from a
monochromatic interval of order less than $\tau_0n$. Hence
\eqref{eq:case2intervals} gives $\sum_{P\in\mathcal P_{\mathrm{s}}}|P|
\le2\beta n+3\alpha^2n+30h\le\sigma n$. Since, apart from $O_h(1)$ branch and transition vertices, all vertices of $S_X\cup S_Y$ lie in the intervals counted in \eqref{eq:case2intervals}, the hierarchy and
$|X|,|Y|\ge(1/2-\eta)n$ give $|Z_X|,|Z_Y|\ge0.49n$.
Consequently, each of $\mathcal P_X$ and $\mathcal P_Y$ contains a
subpath of order at least $\tau_b n$; otherwise
\eqref{eq:case2-task-sums} would give $|Z_X|\le\sigma n$ or
$|Z_Y|\le\sigma n$, respectively.

\medskip
\noindent\textbf{Step 5. Completing the remaining subpaths.}
\medskip

Define $C'_X=C_X\cap Z_X$, $E'_X=D_B\cap Z_X$, $C'_Y=C_Y\cap Z_Y$, and $E'_Y=D_A\cap Z_Y$. The endpoints of all remaining subpaths lie in the corresponding cores. By \eqref{eq:case2intervals} and the hierarchy,
we have $|C'_X|,|C'_Y|\ge n/3$ and $|E'_X|,|E'_Y|\le 2\varepsilon n/\eta$.

Moreover, for $Z\in\{X,Y\}$,
\begin{equation}\label{eq:case2degrees}
 |C'_Z\setminus N(c)|\le3\eta n
 \qquad(c\in C'_Z).
\end{equation}

Since all choices in Steps~2 and~3 avoided $R_X\cup R_Y$, every vertex
of $E'_X$ still has its two reserved neighbors in $C'_X$, and the
analogous statement holds for $E'_Y$.

We complete the two parts separately. Fix $\Gamma\in\{X,Y\}$ and write
$C=C'_\Gamma$ and $E=E'_\Gamma$. Let $P_1,\ldots,P_q$ be the remaining subpaths assigned to this part. For each $i\in[q]$, let $r_i=|P_i|$ and denote the endpoints of $P_i$ by $x_i$ and $y_i$. By Step~4, some $P_{i_*}$ has $r_{i_*}\ge\tau_b n$. Set $u=x_{i_*}$ and $v=y_{i_*}$.

List $E=\{w_1,\ldots,w_t\}$, and write
$a_i=a_{w_i}$ and $b_i=b_{w_i}$ for the two reserved core neighbors
of $w_i$. Set $b_0=u$. For each $i\in[t]$, choose a fresh common
neighbor $c_i\in C$ of $b_{i-1}$ and $a_i$, avoiding the endpoints
of the remaining subpaths and all reserved vertices. Concatenating the
paths with vertex sequences $(b_{i-1},c_i,a_i,w_i,b_i)$, $i\in[t]$, gives a path $Q$ from $u$ to $b_t$ containing every vertex of $E$. The vertices $c_i$ can be chosen greedily by \eqref{eq:case2degrees}, since any two vertices of $C$ have at least $|C|-8\eta n$ common neighbors in $C$. If $t=0$, let $Q$ consist only of $u$ and set $b_t=u$. In either case,
$|Q|=1+4t=o(\tau_b n)$.

We first complete all remaining subpaths with $r_i<\tau_b n$ inside
the unused core. For such a subpath with endpoints $x_i,y_i$, choose
$r_i-3$ internal vertices successively, and then choose the last
internal vertex as a fresh common neighbor in $C$ of the current
endpoint and $y_i$. Throughout this process, at most
$\sigma n+o(\eta n)+O_h(1)$ vertices are unavailable, including the
reserved vertices not used on $Q$. Since $|C|\ge n/3$, every vertex
of $C$ has at most $4\eta n$ non-neighbors in $C$, and any two
vertices of $C$ therefore have at least $|C|-8\eta n$ common
neighbors in $C$. 

Let $\mathcal L$ be the set of indices $i$ for which the subpath of
order $r_i$ is still unembedded. Thus $r_i\ge\tau_b n$ for every
$i\in\mathcal L$, and $i_*\in\mathcal L$. After deleting the vertices already used by $Q$ and by the completed short subpaths, remove also the endpoints of the subpaths indexed by $\mathcal L$ from the remaining vertices of $C$, and denote the resulting set by $C_0$. We shall partition $C_0$ into sets $B_i$, $i\in\mathcal L$,
with
$$
|B_i|=r_i-2\quad\text{for }i\ne i_*,
\qquad
|B_{i_*}|=r_{i_*}-|Q|-1.
$$

The remaining subpaths, including their already embedded endpoints,
account for all vertices of $C\cup E$. Hence
$|C|=\sum_{i=1}^q r_i-t$. The path $Q$ contains $|Q|-t$ vertices of
$C$, while the subpaths already completed use
$\sum_{i\notin\mathcal L}r_i$ vertices of $C$. Among the
$2|\mathcal L|$ endpoints of the remaining subpaths, only $u$
already lies on $Q$. Therefore
$$
|C_0|
=\sum_{i\in\mathcal L}r_i-|Q|-2|\mathcal L|+1
=\sum_{i\in\mathcal L\setminus\{i_*\}}(r_i-2)
 +(r_{i_*}-|Q|-1).
$$
Thus the prescribed sizes sum to $|C_0|$. Since
$r_i\ge\tau_b n$ for every $i\in\mathcal L$ and
$|Q|=o(\tau_b n)$, each $B_i$ has order at least
$\tau_b n/2$ for sufficiently large $n$.

Set $N_0=|C_0|$. Since every $B_i$ has order at least
$\tau_b n/2$ and $N_0\le n$, we have
$N_0\ge\tau_b n/2$ and $|B_i|\ge(\tau_b/2)N_0$ for every
$i\in\mathcal L$. Let $W$ consist of $C_0$, the endpoints of the
subpaths indexed by $\mathcal L$, and $b_t$. All vertices of $W$ lie in the core. Hence \eqref{eq:case2degrees} gives
$|C_0\setminus N(z)|\le4\eta n\le(8\eta/\tau_b)N_0$
for every $z\in W$. Moreover, $|\mathcal L|\le K$ and, for sufficiently large $n$, $|W|\le N_0+2K+1\le2N_0$. Hence, by
\cref{allocation} and the hierarchy, $C_0$ can be partitioned into
the prescribed sets $B_i$, $i\in\mathcal L$, so that every
$z\in W$ has fewer than $|B_i|/10$ non-neighbors in each $B_i$.

For $i\in\mathcal L\setminus\{i_*\}$, every vertex of
$G[B_i\cup\{x_i,y_i\}]$ has degree at least $\frac{9|B_i|}{10}-1>\frac{|B_i|+3}{2}$. Since this graph has order $|B_i|+2$,  \cref{hpaths}\textup{(i)} yields a Hamilton $x_i$--$y_i$ path. Its order is $|B_i|+2=r_i$, as required.

The same argument applied to
$G[B_{i_*}\cup\{b_t,v\}]$ gives a Hamilton $b_t$--$v$ path
containing all vertices of $B_{i_*}$. Concatenating this path with
$Q$ gives a $u$--$v$ path of order $|Q|+|B_{i_*}|+1=r_{i_*}$.

Thus all remaining subpaths in the chosen part are completed with
their prescribed orders and together cover the whole part. Repeating
the argument for the other part completes the embedding.

Together with the previously embedded vertices and transition edges,
these paths form a color-preserving copy of $F$. Since $|V(F)|=n$, this copy is spanning, and the path replacing $e_i$ has length $\ell_i$ for every
$i\in[h]$. This proves the theorem in Extremal Case~2.
\end{proof}
\medskip
\begin{proof}[Proof of \cref{thm:main}]
If $G$ is $\varepsilon$-non-extremal, the result follows from \cref{sc3}. Otherwise $G$ lies in Extremal Case~1 or Extremal Case~2, and the result follows from \cref{subsec:extremal1} or \cref{subsec:extremal2}, respectively.
\end{proof}

\section{Concluding remarks}\label{sc5}

Two questions seem particularly natural. First, the restriction on the
total length of the short subdivision paths is used only to keep the
part embedded before the long paths sufficiently small. It would be
interesting to know whether it can be removed.

\begin{question}\label{ques51}
Does \cref{thm:main} remain true for arbitrary prescribed lengths
$\ell_1,\ldots,\ell_h\ge4$ satisfying
$\sum_{i=1}^h\ell_i=n-|V(H)|+h$?
\end{question}

The proof also suggests that the correct additive term should depend on
the structure of $H$ rather than only on its number of edges.

\begin{question}\label{ques52}
For a fixed graph $H$, determine the smallest integer $c(H)$ such that
$\delta(G)\ge n/2+c(H)$ guarantees every admissible prescribed
spanning $H$-subdivision for all sufficiently large $n$. 
\end{question}


\begin{thebibliography}{99}

\bibitem{babu-diwan}
Ch. S. Babu and A. A. Diwan,
Subdivisions of graphs: A generalization of paths and cycles,
\emph{Discrete Math.} 308 (2008), 4479--4486.





\bibitem{cheng-wang-yan-linked}
Y. Cheng, Z. Wang and J. Yan,
A Dirac-type theorem for arbitrary Hamiltonian $H$-linked digraphs,
\emph{arXiv:2401.17475}.

\bibitem{chizmar2016}
E. Chizmar, C. Magnant and P. Salehi Nowbandegani,
Note on semi-linkage with almost prescribed lengths in large graphs,
\emph{Graphs Combin.} 32 (2016), 881--886.

\bibitem{coll2019}
V. E. Coll, C. Magnant and P. Salehi Nowbandegani,
Degree sum and graph linkage with prescribed path lengths,
\emph{Discrete Appl. Math.} 257 (2019), 85--94.


\bibitem{dirac1952}
G. A. Dirac,
Some theorems on abstract graphs,
\emph{Proc. London Math. Soc.} (3) 2 (1952), 69--81.

\bibitem{gilfernandez2023}
I. Gil Fern\'andez, J. Hyde, H. Liu, O. Pikhurko and Z. Wu,
Disjoint isomorphic balanced clique subdivisions,
\emph{J. Combin. Theory Ser. B} 161 (2023), 417--436.


\bibitem{Khan}
I. Khan,
\emph{Spanning subgraphs in graphs and hypergraphs},
Ph.D. thesis, Rutgers University, 2011.

\bibitem{kim-length-2026}
J. Kim, H. Liu, Y. Tang, G. Wang, D. Yang and F. Yang,
Extremal density for subdivisions with length or sparsity constraints,
\emph{J. Combin. Theory Ser. B} 177 (2026), 67--104.






\bibitem{Lee}
H. Lee,
Spanning subdivisions in dense digraphs,
\emph{European J. Combin.} 124 (2025), Article 104059.



\bibitem{liu-montgomery}
H. Liu and R. Montgomery,
A solution to Erd\H{o}s and Hajnal's odd cycle problem,
\emph{J. Amer. Math. Soc.} 36 (2023), 1191--1234.

\bibitem{luan2023}
B. Luan, Y. Tang, G. Wang and D. Yang,
Balanced subdivisions of cliques in graphs,
\emph{Combinatorica} 43 (2023), 885--907.



\bibitem{Pavez}
M. Pavez-Sign\'e,
Spanning subdivisions in Dirac graphs,
\emph{Combin. Probab. Comput.} 33 (2024), 121--128.

\bibitem{lee-pavez-petrov}
M. Pavez-Sign\'e, H. Lee and T. Petrov,
Spanning clique subdivisions in pseudorandom graphs,
\emph{Combin. Probab. Comput.} (2026), 1--15.


\bibitem{thomassen1984}
C. Thomassen,
Subdivisions of graphs with large minimum degree,
\emph{J. Graph Theory} 8 (1984), 23--28.

\bibitem{wang2023}
Y. Wang,
Balanced subdivisions of a large clique in graphs with high average degree,
\emph{SIAM J. Discrete Math.} 37 (2023), 1262--1274.

\bibitem{wang-cheng-yan2026}
Z. Wang, Y. Cheng and J. Yan,
Spanning $H$-subdivisions and perfect $H$-subdivision tilings in dense
digraphs,
\emph{Combin. Probab. Comput.} (2026), 1--31.

\end{thebibliography}
\end{document}